\documentclass{amsart}
\usepackage{latexsym,amscd,amssymb,url}
\usepackage{extarrows}
\usepackage{verbatim}
\usepackage{mathtools}
\usepackage{comment}
\usepackage{appendix}
\usepackage[shortlabels]{enumitem}

\usepackage{bbm}
\usepackage{dsfont}
\usepackage{tikz-cd}
\usepackage[all,cmtip]{xy}  
\usepackage{graphicx}
\usepackage{xcolor}
\usepackage{enumitem} 
\definecolor{dblue}{rgb}{0,0,.6}
\usepackage[colorlinks=true, linkcolor=dblue, citecolor=dblue, filecolor = dblue, menucolor = dblue, urlcolor = dblue]{hyperref}

\numberwithin{equation}{section}

\newtheorem{theorem}{Theorem}[section]

\theoremstyle{plain}

\newtheorem{question}[theorem]{Question}

\newtheorem{corollary}[theorem]{Corollary}

\newtheorem{definition}[theorem]{Definition}
\newtheorem{example}[theorem]{Example}

\newtheorem{lemma}[theorem]{Lemma}

\newtheorem{proposition}[theorem]{Proposition}
\newtheorem{remark}[theorem]{Remark}

\newcommand{\del}{\partial}
\newcommand{\Z}{\mathbb Z}
\newcommand{\Q}{\mathbb Q}

\newcommand{\C}{\mathbb C}
\newcommand{\N}{\mathbb N}
\newcommand{\R}{\operatorname{R}}

\newcommand{\CP}{\mathbb P}

\newcommand{\im}{\operatorname{im}}

\newcommand{\Hom}{\operatorname{Hom}}

\newcommand{\id}{\operatorname{id}}

\newcommand{\Spec}{\operatorname{Spec}}

\newcommand{\Gal}{\operatorname{Gal}}

\newcommand{\CH}{\operatorname{CH}}

\newcommand{\cl}{\operatorname{cl}}

\newcommand{\F}{\mathbb F}

  \newcommand{\Griff}{\operatorname{Griff}}

\newcommand{\alg}{\operatorname{alg}}

\newcommand{\et}{\text{\'et}}
\newcommand{\zar}{\text{Zar}}
\newcommand{\proet}{\text{pro\'et}}
\newcommand{\Ab}{\operatorname{Ab}}

\newcommand{\colim}{\operatorname{colim}}

\newcommand{\nr}{\operatorname{nr}}
\newcommand{\sep}{\operatorname{sep}}
\newcommand{\cons}{cons}

\newcommand{\dashedlongrightarrow}{\xymatrix@1@=15pt{\ar@{-->}[r]&}}
\renewcommand{\longrightarrow}{\xymatrix@1@=15pt{\ar[r]&}}
\renewcommand{\mapsto}{\xymatrix@1@=15pt{\ar@{|->}[r]&}}
\renewcommand{\twoheadrightarrow}{\xymatrix@1@=15pt{\ar@{->>}[r]&}}
\newcommand{\hooklongrightarrow}{\xymatrix@1@=15pt{\ar@{^(->}[r]&}}
\newcommand{\congpf}{\xymatrix@1@=15pt{\ar[r]^-\sim&}}
\renewcommand{\cong}{\simeq}

\usepackage{xcolor}

\begin{document}

%\title[Divisibility properties of some motivic invariants]{Divisibility properties of some motivic invariants}
%\title[Divisibility properties of some motivic invariants]{Divisibility properties of unramified motivic cohomology and the motivic Bloch--Ogus spectral sequence}
%\title[Divisibility properties of some motivic invariants]{Divisibility properties of some motivic invariants in motivic Bloch--Ogus theory} 
\title[Divisibility phenomena in motivic Bloch--Ogus theory]{Divisibility phenomena in motivic Bloch--Ogus theory}

\author{Jean-Louis Colliot-Thélène} 
\address{Universit\'e Paris-Saclay, CNRS, Laboratoire de math\'ematiques d’Orsay, 91405, Orsay, France}
\email{jean-louis.colliot-thelene@universite-paris-saclay.fr} 

\author{Stefan Schreieder} 
\address{Institute of Algebraic Geometry, Leibniz University Hannover, Welfengarten 1, 30167 Hannover, Germany.}
\email{schreieder@math.uni-hannover.de}

\date{\today}  
% MSC 2020 classifications used:
% 14C25  Algebraic cycles
% 19E15  Algebraic cycles and motivic cohomology (K-theoretic aspects)
% 14F42  Motivic cohomology; motivic homotopy theory
% 14F43  Other algebro-geometric (co)homologies
%           (e.g., intersection, equivariant, Lawson, Deligne (co)homologies)
% 14C35 -- Applications of methods of algebraic K-theory in algebraic geometry
% 
%  14C15   (Equivariant) Chow groups and rings; motives 
%  14F20   Étale and other Grothendieck topologies and (co)homologies


\begin{abstract} 
Let $X$ be a smooth projective variety over a field $k$.
For $k$ separably closed, we prove that the subgroup of unramified classes in the Milnor K-group $K^M_i(k(X))$ of the function field of $X$ is contained in the subgroup of $n$-divisible elements of $K^M_i(k(X))$ for any integer $n$ invertible in $k$.  
This generalizes to a statement for unramified motivic cohomology of arbitrary bidegree. 
We further show that whenever $k$ is finite or separably closed and $\ell$ is a prime invertible in $k$, then all but the last step in the Bloch--Ogus filtration of the motivic cohomology of $X$ are $\ell$-divisible up to torsion.  
Generalizations of this last result to arbitrary quasi-projective $k$-schemes are also proven.
 \end{abstract}

\maketitle 

\tableofcontents

\section{Introduction}

Let $X$ be a smooth projective variety over a separably closed field $k$ and consider the Chow group $\CH^i(X)$ of codimension $i$ cycles modulo rational equivalence.
For $i=0,1,\dim X$, this group is the extension of a finitely generated group by a divisible group.
Schoen \cite{schoen-modn} showed that this fails in general for all other codimensions $1< i< \dim X $; important refinements and generalizations of this result appeared in \cite{RS,totaro-chow,diaz,Sch-griffiths,alexandrou,scavia,alexandrou-zhou}, see also Appendix \ref{sec:Chow-tensor-Q/Z} below.
All those results rely on work of Bloch--Esnault \cite{BE}, where the first example of a cycle that is homologically trivial but not divisible was discovered.

The purpose of this paper is to show, to the contrary, that various natural motivic invariants, especially those associated to motivic Bloch--Ogus theory, do in fact satisfy somewhat surprising divisibility phenomena.
 
\subsection{Unramified Milnor K-theory}
For a ring $R$,  we denote by $K^M_i(R)$ the quotient of the free tensor algebra $T^\ast(R^\ast)$ on the units of $R$, by the two-sided ideal generated by $a\otimes (1-a)$ with $a\in R\setminus \{0,1\}$, see \cite[Definition 2.1]{kerz}.
For a smooth $k$-scheme $X$, we then define the Milnor K-theory sheaf  $\mathcal K^M_i$ as the sheaf associated to the presheaf which maps an open subset $U\subset X$ to $K^M_i(\mathcal O(U))$.
Similarly, the Milnor K-theory sheaf modulo some integer $n$ is defined as 
$\mathcal K^M_i/n\coloneq \mathcal K^M_i \otimes \Z/n$.
Both sheaves play a crucial role in the Gersten conjecture and in Bloch--Ogus theory  for Milnor K-theory, see \cite{BO,kerz}.

\begin{theorem} \label{thm:unramified-Milnor-divisible}
Let $X$ be a smooth projective variety over a separably closed field $k$ and let $n$ be an integer invertible in $k$.
For $i\geq 1$, the natural reduction map $H^0(X,\mathcal K^M_i)\to H^0(X,\mathcal K^M_i/n)$ is zero. 
\end{theorem}

The case $i=1$ in the above theorem is easy because $H^0(X,\mathcal K^M_1)=k^\ast$; the case $i=2$ is due to Colliot-Th\'el\`ene--Raskind \cite[Corollary 1.7]{CTR}. 

By the Gersten conjecture for Milnor K-theory,  proven by Kerz \cite{kerz},  $H^0(X,\mathcal K^M_i)\cong K^M_i(k(X))_{nr}$ agrees with the subgroup of unramified classes of the Milnor K-group $K^M_i(k(X))$.
Similarly,  $H^0(X,\mathcal K^M_i/n) $ can be identified with the subgroup of $K^M_i(k(X))/n$ of classes that are unramified over $k$.
We thus obtain the following result.

\begin{corollary} \label{cor:unramified-Milnor-divisible} 
In the notation of Theorem \ref{thm:unramified-Milnor-divisible}, the subgroup $ K^M_i(k(X))_{nr} \subset K^M_i(k(X))$ of classes that are unramified over $k$ is contained in the maximal $n$-divisible subgroup: 
$$ 
K^M_i(k(X))_{nr}\subset K^M_i(k(X))_{n-\operatorname{div}}.
$$ 
\end{corollary}

Recall the natural map $K^M_i(k(X))\to H^i(k(X),\mu_{n}^{\otimes i})$, which is surjective by the Bloch--Kato conjecture, proven by Rost and Voevodsky \cite{Voevodsky}.
Despite this surjection, the above corollary shows that for any integer $n\geq 2$ invertible in $k$, a nonzero class in the unramified cohomology group $ H^i_{nr}(X_{\et},\mu_{n}^{\otimes i})$ can never be lifted to an unramified class in $K^M_i(k(X))$.
For more details on unramified cohomology and interesting examples of unramified classes, we refer to the surveys \cite{CT,Sch-survey}.

\subsection{Unramified motivic cohomology} 
Let $X$ be a smooth equi-dimensional algebraic scheme over a field $k$. 
The motivic cohomology of $X$ with values in an abelian group $A$ is defined as the hypercohomology 
$$
H^i_M(X,A(n))\coloneq  H^i(X_\zar, A(n)) 
$$
of Bloch's cycle complex $A(n)=\Z(n)\otimes^{\mathbb L} _{\Z}A \in D(X_{\zar})$ with values in $A$, cf.\ Section \ref{subsec:motivic-cohomology} below.  
For $A=\Z$, these groups agree canonically with Bloch's higher Chow groups:
$$
H^i_M(X,\Z(n))\cong \CH^n(X,2n-i),
$$ 
see \cite[p.\ 269, (iv)]{bloch-motivic} and \cite{bloch-JAG}. 
In particular, $H^{2i}_M(X,\Z(i))\cong \CH^i(X)$ agrees with ordinary Chow groups.

Assume now that $X$ is irreducible.
We may then consider the motivic cohomology $H^i_M(k(X),A(n))\coloneq H^i_M(\Spec k(X),A(n))$ of the function field $k(X)$; 
this agrees with $\colim_U H^i_M(U,A(n))$, where $U$ runs through all dense open subsets of $X$. 
The  unramified motivic cohomology of $X$ is then given by
\begin{align} \label{def:H_nr-motivic}
H^i_{M,nr}(X,A(n))\coloneq \{\alpha\in H^i_M(k(X),A(n))\mid \del_x\alpha=0 \quad \text{for all $x\in X^{(1)}$}\} ,
\end{align}
where $\del_x\alpha\in H^{i-1}_M(\kappa(x),A(n-1))$ denotes the residue of $\alpha$ at $x$,
see Section \ref{subsec:H_M,nr} below for more details. 

\begin{theorem} \label{thm:H_M-unramified-divisible}
    Let $k$ be a separably closed field and let $i,n$ be integers with $n\geq 1$.
    Let $X$ be a smooth projective variety over $k$ and let $m$ be an integer invertible in $k$.
    Then the natural maps
    $$
    H^i_{M,nr}(X,\Z(n))\longrightarrow H^i_{M,nr}(X,\Z/m(n)) \quad \text{and}\quad 
    H^i_{M,nr}(X,\Z(n))\longrightarrow H^i_{nr}(X_\et,\mu_{m}^{\otimes n}) 
    $$
    are zero.
\end{theorem}

The second map in the above theorem is induced by the natural comparison map $H^i_{M,nr}(X,\Z/m(n))\to  H^i_{nr}(X_\et,\mu_{m}^{\otimes n}) $ from the Zariski to the \'etale site, see \cite[Theorem 1.5]{geisser-levine}. 
In fact, as a consequence of the Beilinson--Lichtenbaum conjectures, proven by Rost and Voevodsky \cite{Voevodsky}, we have $H^i_{M,nr}(X,\Z/m(n))\cong H^i_{nr}(X_\et,\mu_{m}^{\otimes n})$ if $i\leq n$ and $H^i_{M,nr}(X,\Z/m(n))=0$ otherwise, see Lemma \ref{lem:H_M,nrZ/m}.

A strengthening of Theorem \ref{thm:H_M-unramified-divisible} in the case $i\neq n$ is proven in Theorem \ref{thm:strengthen-thm:H_M-unramified-divisible} below.

\subsection{Motivic coniveau filtration} 
Let  $\mathcal H_M^i(\Z(n))$ denote the Zariski sheaf associated to the presheaf $U\mapsto H^i_M(U,\Z(n))$.
Following \cite{BO}, we have the motivic coniveau spectral sequence
$$
E_2^{p,q}=H^p(X,\mathcal H_M^q(\Z(n))) \Longrightarrow H^{p+q}_M(X,\Z(n)) ,
$$  
see \cite[p.\ 269, (iv)]{bloch-motivic}.
This spectral sequence induces the Bloch--Ogus filtration
\begin{align} \label{eq:filtration-L}
\cdots \subset  L_0\subset L_1\subset \dots \subset L_{n-1}\subset L_n=H^i_M(X,\Z(n)) ,
\end{align} 
given by 
$$
L_jH^i_M(X,\Z(n))\coloneq \im(H^i(X_\zar, \tau_{\leq j}\Z(n)) \to H^i(X_\zar, \Z(n)) ) .
$$
By Lemma \ref{lem:L-versus-N}, 
$$
L_{j}H^i_M(X,\Z(n))=N^{i-j}H^i_M(X,\Z(n))
$$ 
agrees with the coniveau filtration, where  $N^cH^i_M(X,\Z(n))$ consists of all classes of $H^i_M(X,\Z(n))$ that vanish outside a closed subset of codimension at least $c$.
For $i\leq 2n$, this filtration is of the form
\begin{align} \label{eq:filtration-N-intro}
H^i_M(X,\Z(n))=N^{i-n}H^i_M(X,\Z(n)) \supset  N^{i-n+1}\supset  N^{i-n+2} \supset   \dots \supset N^{n}\supset N^{n+1}=0,
\end{align}
see Lemma \ref{lem:N^{i-n}} below.
From this description we see that the filtration is of length $2n+1-i$, hence it is trivial if  $i=2n$, but it is interesting in general.\footnote{Note that for $i=2n$ the filtration $N^\ast$ does  not coincide with Bloch's coniveau filtration on $\CH^n(X)$, which measures the dimension of a closed subset on which a given class is homologically trivial, see \cite{bloch-coniveau} and \cite[\S 1.1]{Sch-refined}.}

\begin{theorem} \label{thm:filtration-L-intro}
Let $k$ be either a finite field or a separably closed field and let $X$ be a smooth equi-dimensional quasi-projective scheme over $k$. 
Then,  for all primes $\ell$ invertible in $k$, we have
\begin{align} \label{eq:L^jH^iotimesQl/Zl}
L_jH^i_M(X,\Z(n))\otimes_\Z \Q_\ell/\Z_\ell=0\quad \text{ for $j<n$;}
\end{align} 
\begin{align} \label{eq:L^jH^iotimesQl/Zl-N}
N^cH^i_M(X,\Z(n))\otimes_\Z \Q_\ell/\Z_\ell=0\quad \text{ for $c>i-n$.}
\end{align} 
\end{theorem}

Note that the above result covers all but the last (resp.~first) filtration step in \eqref{eq:filtration-L} (resp.~\eqref{eq:filtration-N-intro}).  

In the case of smooth projective varieties over a finite field,  Parshin's conjecture on the algebraic K-theory of smooth projective schemes over finite fields predicts that 
$H^i_M(X,\Z(n))\otimes \Q_\ell$ is zero for $i\neq 2n$, see e.g.~\cite[pp.~189-190]{jannsen-book}.  
The above theorem proves the weaker assertion that, for $j<n$, the subgroup $L_jH^i_M(X,\Z(n))$ vanishes after tensoring with $\Q_\ell/\Z_\ell$.  

Recall that an abelian group $G$ satisfies $G\otimes \Q/\Z=0$ if and only if $G\otimes \Q_\ell/\Z_\ell=0$ for all primes $\ell$.
Moreover,  $G\otimes \Q_\ell/\Z_\ell=0$ is equivalent to asking that $G$ is $\ell$-divisible up to torsion, i.e.\ it is the extension of an $\ell$-divisible group by an $\ell$-primary torsion group, see Lemma \ref{lem:A-tensor-Q/Z-elementary} below. 
This is a restrictive property; for instance, such groups do not admit non-trivial maps to $\Z$.

By work of Kerz, for smooth varieties over an infinite ground field $k$, there is a canonical isomorphism $\mathcal K_n^M \stackrel{\cong}\to \mathcal H^n(\Z(n))$, where  $\mathcal K^M_n$ denotes the  Milnor K-theory sheaf on $X_\zar$, see \cite[Theorem 1.1]{kerz}. 
Since $\mathcal H^j(\Z(n))=0$ for $j>n$, the hypercohomology spectral sequence induces an edge map $H^i_M(X,\Z(n))\to H^{i-n}(X,\mathcal K_n^M)$ and we consider its image 
$$
H^{i-n}(X,\mathcal K_n^M)_\infty \coloneq \im (H^i_M(X,\Z(n))\longrightarrow H^{i-n}(X,\mathcal K_n^M)).
$$
Theorem \ref{thm:filtration-L-intro} proves that all but this last filtration step of $L_\ast H^i_M(X,\Z(n))$ are $\ell$-divisible up to torsion and so we obtain the following corollary.

\begin{corollary} \label{cor:filtration-L-intro-1}
Let $X$ be a smooth quasi-projective equi-dimensional scheme over a separably closed field $k$. 
Then the natural map
$$
H^i_M(X,\Z(n))\otimes_\Z \Q_\ell/\Z_\ell
\stackrel{\cong} \longrightarrow
H^{i-n}(X,\mathcal K_n^M)_\infty \otimes_\Z \Q_\ell/\Z_\ell
$$
is an isomorphism for all primes $\ell$ invertible in $k$. 
In particular, $H^i_M(X,\Z(n))\otimes_\Z \Q_\ell/\Z_\ell =0$ for $i<n$. 
\end{corollary}

In the body of the paper, we prove a version of Theorem \ref{thm:filtration-L-intro} which does not require the smoothness assumption on $X$.
To state the result,  for any equi-dimensional quasi-projective scheme $X$ over $k$ we denote by $H^i_{BM,M}(X,\Z(n))$ the hypercohomology of Bloch's cycle complex in the Zariski topology.
In other words, if $d_X=\dim X$, then
\begin{align} \label{eq:H^i-BM,M-intro}
H^i_{BM,M}(X,\Z(n))=H^{BM,M}_{2d_X-i}(X,\Z(d_X-n))\cong \CH_{d_X-n}(X,2n-i)
\end{align}
agrees with motivic Borel--Moore homology (cf.~\cite[\S 1.1]{levine}); we use the cohomological indexing for convenience, as it makes various statements and arguments in this paper independent of the dimension of $X$.
If $X$ is smooth and equi-dimensional, then $H^i_{BM,M}(X,\Z(n))=H^i_{M}(X,\Z(n))$. 
We further define $N^jH^i_{BM,M}(X,\Z(n))\subset H^i_{BM,M}(X,\Z(n))$ as the subgroup of classes that vanish outside a closed subset of codimension at least $j$.
Then we have the following generalization of Theorem \ref{thm:filtration-L-intro}.

\begin{theorem} \label{thm:filtration-intro-BM}
Let $k$ be a field that is either separably closed or finite and let $\ell$ be a prime invertible in $k$.
Let $X$ be an equi-dimensional quasi-projective variety over $k$.
Then the following holds:
\begin{enumerate}
\item $N^cH^i_{BM,M}(X,\Z(n))\otimes \Q_\ell/\Z_\ell=0$ for $c>i-n$;\label{item:thm:Hi_MXQ/Z-vanishing-finite:1}
\item $H^i_{BM,M}(X,\Z(n))\otimes \Q_\ell/\Z_\ell=0$ for $i<n$.\label{item:thm:Hi_MXQ/Z-vanishing-finite:2}
\end{enumerate} 
\end{theorem} 

Using the localization exact sequence, item \eqref{item:thm:Hi_MXQ/Z-vanishing-finite:1} can be deduced from item \eqref{item:thm:Hi_MXQ/Z-vanishing-finite:2}.
The latter will in turn be proven via careful weight arguments, combined with consequences of the proof of the Beilinson--Lichtenbaum conjecture, due to Rost and Voevodsky \cite{Voevodsky}. 
 
For arbitrary smooth quasi-projective varieties, the vanishing results in
Theorems \ref{thm:filtration-L-intro}, \ref{thm:filtration-intro-BM} and
Corollary \ref{cor:filtration-L-intro-1} are sharp,   see Examples \ref{example:sharp-1} and \ref{example:sharp-2} below. 

  It is natural to ask whether these results are sharp for smooth projective varieties as well.
Certainly, the existence of a (non-canonical) degree map on cycles in $H_M^{2n}(X,\Z(n))=\CH^n(X)$ shows that at least for $i=2n\leq 2\dim X$, the condition $j<n$ is necessary in \eqref{eq:L^jH^iotimesQl/Zl}. 
More generally, based on \cite{BE,schoen-modn,RS,totaro-chow}, it turns out that for any $n\geq 3$, there is a smooth complex projective variety of dimension $n$ such that for any integer $2\leq i\leq n-1$ and any subgroup $M\subset \CH^i(X)$ with finitely generated cokernel, $M\otimes \Q_\ell/\Z_\ell\neq 0$ for all primes $\ell$, see Corollary \ref{cor:Chow-tensor-Ql/Zl} in Appendix \ref{sec:Chow-tensor-Q/Z}.   

\begin{question} \label{question:intro}
Let $k$ be either a finite field or a separably closed field and let $X$ be a smooth projective variety over $k$. 
Is it true that the group $H^i_M(X,\Z(n))\otimes_\Z \Q_\ell/\Z_\ell$ from Corollary \ref{cor:filtration-L-intro-1} vanishes for $i\neq 2n$?
\end{question}

In Appendix \ref{sec:appendix-B}, 
we show that the above question has a positive answer for special values of $(i,n)$, see Propositions \ref{prop:appendix:separable-closed} and \ref{prop:appendix-finite-field} below.
These results rely on divisibility results for Lichtenbaum motivic cohomology, as studied by Geisser, Kahn, Rosenschon--Srinivas and others, see e.g.\ \cite{geisser,kahn,RS-JIMJ}.
As a consequence, Question \ref{question:intro} has a positive answer whenever the \'etale cycle class map 
\begin{align} \label{eq:cl-injective}
\cl\colon H^i_M(X,\Q_\ell/\Z_\ell(n))\longrightarrow H^i(X_\et,\Q_\ell/\Z_\ell(n)) 
\end{align}
is injective.
By work of Suslin--Voevodsky \cite{suslin-voevodsky}, Geisser--Levine \cite{geisser-levine}, and the proof of the Beilinson--Lichtenbaum conjecture due to Rost and Voevodsky \cite{Voevodsky},  
{\it this holds for instance for $i\leq n+1$}.

In contrast, we note that for $i\geq n+2$, this injectivity  fails in general.
In fact, the injectivity of \eqref{eq:cl-injective} implies that $H^i_M(X,\Q_\ell/\Z_\ell(n))$ is of cofinite type, i.e.\ a product of a finite group with finitely many copies of $\Q_\ell/\Z_\ell$.
However, for $k=\C$ and $i\geq n +2$, the group $H^i_M(X,\Q_\ell/\Z_\ell(n))$ is in general not of cofinite type: This follows for instance from \cite[Corollary 1.2]{alexandrou-zhou} together with the Bockstein sequence for motivic cohomology; similar results for Chow groups were previously proven in  \cite{schoen-modn,RS,totaro-chow,diaz,Sch-griffiths,alexandrou,scavia}.

\subsection{Structure of the paper}  
We fix notation and collect some preliminary material in Section \ref{sec:prelim}.

In Section \ref{sec:mixed-and-ind-mixed}, following work of Deligne \cite{deligneII} (see also \cite{Huber} and \cite{morel}), we discuss the weights of some natural mixed and ind-mixed Galois modules.
This includes in particular the computation of the weights of the ind-mixed Galois modules given by refined unramified $\ell$-adic cohomology (see \cite{Sch-refined}).  
We also collect some consequences for the cohomology and refined unramified cohomology of quasi-projective schemes over finite fields.

In Section \ref{sec:proof-of-thm:H_M-unramified} we discuss the notion of refined unramified motivic cohomology, following \cite{Sch-refined}.
We compare this to $\ell$-adic refined unramified cohomology and use the aforementioned weight computations to prove Theorems \ref{thm:unramified-Milnor-divisible} and \ref{thm:H_M-unramified-divisible}, and Corollary \ref{cor:unramified-Milnor-divisible}. 

In Section \ref{sec:thm:filtration-L-intro} we prove Theorems \ref{thm:filtration-L-intro} and \ref{thm:filtration-intro-BM} from the introduction. 
This relies on a cycle class map from Borel--Moore motivic cohomology to $\ell$-adic Borel--Moore pro-\'etale homology, see Lemma \ref{lem:cycle-class-H_BM}, together with weight arguments and an extension of the Beilinson--Lichtenbaum conjecture to singular schemes, due to Kok and Zhou \cite{kok-zhou}. 

In Section \ref{sec:factoring-the-map}, we discuss some divisibility phenomena for the Bloch--Ogus groups $H^i(X_\zar,\mathcal H^j_M(\Z(n)))$; our main result in that section is Theorem \ref{thm:Hi-curlyHj}.
This will be used in Section \ref{subsec:thm:unramified-Milnor-divisible-2} to give an alternative proof of Theorem \ref{thm:H_M-unramified-divisible}.

This paper contains three appendices.
In Appendix \ref{sec:Chow-tensor-Q/Z}, we explain how the main result in \cite{totaro-chow} implies that Theorem \ref{thm:filtration-L-intro} fails for $i=2n$ and $j=n\geq 2$ in a strong sense, see Corollary \ref{cor:Chow-tensor-Ql/Zl}.
In Appendix \ref{sec:merkurjev-conjecture}, we use the Bloch--Kato conjecture, as proved by Rost and Voevodsky, to show that the $\ell$-primary torsion subgroup of Milnor K-theory of a field of characteristic different from $\ell$ is $\ell$-divisible, if the field contains all $\ell$-primary roots of unity; the result had been conjectured by Merkurjev in \cite{merkurjev}.
We use the result to give yet another proof of Theorem \ref{thm:unramified-Milnor-divisible} in Section \ref{subsec:thm:unramified-Milnor-divisible-2}.
Finally, in Appendix \ref{sec:appendix-B},  
for smooth projective varieties, we give a concise description of divisibility results for Lichtenbaum motivic cohomology groups 
(following e.g.\ \cite{geisser,kahn,RS-JIMJ}) and deduce 
further positive answers to Question \ref{question:intro}.  

\section{Preliminaries} \label{sec:prelim}

\subsection{Conventions}
An algebraic scheme is a separated scheme of finite type over a field.
A variety is an integral algebraic scheme. 
A finitely generated field is a field that is finitely generated over its prime field.
We say that an algebraic scheme $X$ over a field $k$ has (or admits) a model over $k_0$ if $X\cong X_0\times_{k_0}k$ for some algebraic scheme $X_0$ over $k_0$.

For an abelian group $A$ and a prime $\ell$, we denote the $\ell^r$-torsion subgroup of $A$ by $A[\ell^r]$; moreover, the $\ell$-primary torsion subgroup of $A$ is denoted by $A[\ell^\infty]$. 

If $\tau$ denotes a Grothendieck topology on $X$, then we denote by $\Ab(X_\tau)$ the abelian category of sheaves of abelian groups on $X_\tau$.
We further denote by $D(X_\tau)\coloneq D(\Ab(X_\tau))$ the (unbounded) derived category of sheaves of abelian groups on $X_\tau$.

\subsection{Elementary division properties for abelian groups}

 \begin{lemma}\label{lem:A-tensor-Q/Z-elementary}
 Let $A$ be an abelian group and $\ell$ a prime number.  
 Then the following properties are equivalent:
 \begin{enumerate}
\item  $A \otimes \Q_{\ell}/\Z_{\ell} = 0$; \label{item:lem:A-tensor-Q/Z-elementary:1}
 \item $A \subset \ell A+ A[\ell^\infty]$;\label{item:lem:A-tensor-Q/Z-elementary:2}
 \item For all $n>0$,  $A \subset \ell^nA+ A[\ell^\infty]$;\label{item:lem:A-tensor-Q/Z-elementary:3}
\item For all $n>0$, the map $A[\ell^\infty] \to  A/\ell^nA$ is surjective;\label{item:lem:A-tensor-Q/Z-elementary:4}
\item The group $A$ is an extension of an $\ell$-divisible group by an $\ell$-primary torsion group.\label{item:lem:A-tensor-Q/Z-elementary:5}
\end{enumerate}

 \end{lemma}
 \begin{proof}
 Note that $$A \otimes \Q_{\ell}/\Z_{\ell} = A \otimes  \lim_{\substack{\longrightarrow\\ n}}  \Z/\ell^n,$$
 where the maps $\Z/\ell^n \to \Z/\ell^{n+1}$ are given by multiplication by $\ell$.
 Assume \eqref{item:lem:A-tensor-Q/Z-elementary:1}. 
 For any $a\in A$, the image of its class in $A/\ell$ under
 some map $\ell^n : A/\ell A \to  A/\ell^{n+1}$ vanishes. 
 Hence $\ell^n a = \ell^{n+1} b$ for  some $b \in A$ and so $\ell^n (a- \ell b)=0 \in A$.
 Thus \eqref{item:lem:A-tensor-Q/Z-elementary:1} implies \eqref{item:lem:A-tensor-Q/Z-elementary:2}, which is equivalent to \eqref{item:lem:A-tensor-Q/Z-elementary:3}, which is equivalent to \eqref{item:lem:A-tensor-Q/Z-elementary:4}.
 Let $B= A/A[\ell^\infty]$ and assume \eqref{item:lem:A-tensor-Q/Z-elementary:4}.
 Then $B/\ell^n=0$ for all $n>0$.
 The exact sequence 
 $$ 
 0 \longrightarrow A[\ell^\infty]\longrightarrow A \longrightarrow   B \longrightarrow 0 ,
 $$
 thus shows that \eqref{item:lem:A-tensor-Q/Z-elementary:4} implies \eqref{item:lem:A-tensor-Q/Z-elementary:5}.
 Let us finally assume \eqref{item:lem:A-tensor-Q/Z-elementary:5}.
 That is, there is an exact sequence
 $$
 0\longrightarrow T\longrightarrow A\longrightarrow D\longrightarrow 0
 $$
 for an $\ell$-primary torsion group $T$ and an $\ell$-divisible group $D$.
 Since $T\otimes \Q_\ell/\Z_\ell=0$ and $D\otimes \Q_\ell/\Z_\ell=0$, we find that \eqref{item:lem:A-tensor-Q/Z-elementary:1} holds, which concludes the proof of the lemma.
 \end{proof}

 \begin{remark} {\rm
     There is a straightforward analogue of Lemma \ref{lem:A-tensor-Q/Z-elementary} where one replaces $A\otimes \Q_\ell/\Z_\ell=0$  by the hypothesis $A\otimes \Q/\Z=0$.}
 \end{remark}
 
\begin{remark} {\rm
If an abelian group $A$ satisfies  $A \otimes {\Q/\Z}=0$ (i.e.~$A \otimes {\Q_{\ell}/\Z_{\ell}}=0$ for all primes $\ell$),  it need not be an extension of a torsion group  by a divisible group.
Indeed, let $A= \prod_{p} \Z/p$, where $p$ runs through all primes.
Let $\ell$ be a prime. 
For any $n>0$, we have $A \otimes \Z/\ell^n = \Z/\ell$ and the natural inclusion $\Z/\ell^n \hookrightarrow \Z/\ell^{n+1}$ given  by multiplication by $\ell$ induces the map
$A \otimes \Z/\ell^n  \to  A \otimes \Z/\ell^{n+1}$ which is multiplication by $\ell$ on $\Z/\ell$, hence zero.
Taking direct limits, we get $A \otimes \Q_{\ell} /\Z_{\ell}=0$.
However, no nonzero element of $A= \prod_{p} \Z/p$ is divisible by all primes. Thus $A$ contains no divisible subgroup.
If $A$ was an extension of a torsion group by a divisible group, it would thus be a torsion group. But the diagonal element $(1, \dots , 1, \dots) \in A$ is not a torsion element.} 
 \end{remark}

\subsection{$\ell$-adic (pro-)\'etale cohomology}
Let $X$ be a scheme endowed with some Grothendieck topology $\tau$ on $X$.
Important examples are the small \'etale site $\tau=\et$ and the small pro-\'etale site $\tau=\proet$, see \cite{BS}.
We denote by $\nu:X_\proet\to X_\et$ the natural map of sites.

For a prime $\ell$ invertible on $X$, we will write
$$
H^i(X_\et,\Z/\ell^r(n))\coloneq H^i(X_\et,\mu_{\ell^r}^{\otimes n})\coloneq \R^i\Gamma(X_\et,\mu_{\ell^r}^{\otimes n})\quad \text{and}\quad H^i(X_\et,\Q_\ell/\Z_\ell(n))\coloneq \lim_{\substack{\longrightarrow\\r}}H^i(X_\et,\Z/\ell^r(n)).
$$
Sometimes we also write $H^i_\et(X,A(n))$ instead of $H^i(X_\et,A(n))$.
By \cite[Proposition 5.2.6.(2)]{BS}, the adjunction map $\id\to \R\nu_\ast \nu^\ast$ is an equivalence and so we may compute the above cohomology groups in the pro-\'etale site as well.

We further consider the sheaves
$$
\widehat \Z_\ell(n)\coloneq \lim_{\substack{\longleftarrow\\ r}} \nu^\ast \mu_{\ell^r}^{\otimes n} \in \Ab(X_\proet)\quad \text{and}\quad \widehat \Q_\ell(n)\coloneq \widehat \Z_\ell(n)\otimes_{\widehat \Z_\ell} \widehat \Q_\ell \in \Ab(X_\proet),
$$
where $\widehat \Q_\ell$ is the sheaf on $X_\proet$ associated to the topological ring $\Q_\ell$ via \cite[Lemma 4.2.12]{BS}. 
Sometimes we will also drop the hat in the above notations; in particular, we will frequently write 
$$
H^i(X_\proet,\Z_\ell(n))\coloneq \R^i\Gamma(X_\proet,\widehat \Z_\ell(n))\quad \text{and}\quad 
H^i(X_\proet,\Q_\ell(n))\coloneq \R^i\Gamma(X_\proet,\widehat \Q_\ell(n)).
$$ 
These groups agree with Jannsen's continuous \'etale cohomology groups from \cite{jannsen}, see \cite[Proposition 5.6.2]{BS}.
Moreover, $H^i(X_\proet,\Q_\ell(n))=H^i(X_\proet,\Z_\ell(n))\otimes_{\Z_\ell} \Q_\ell$.

If $X$ is an algebraic scheme over a field $k$ with $k$ separably closed or finite, then the above groups coincide with the usual \'etale cohomology groups: 
\begin{align} \label{eq:H_proet=H_et}
H^i(X_\proet,\Z_\ell(n))=H^i(X_\et,\Z_\ell(n))\quad \text{and}\quad H^i(X_\proet,\Q_\ell(n))= H^i(X_\et,\Q_\ell(n)) ,
\end{align}
because $H^i(X_\et,\mu_{\ell^r}^{\otimes n})$ is finite in this case and so the Mittag--Leffler condition is satisfied, see \cite[p.\ 208, (0.2)]{jannsen}.
 
 \subsection{$\ell$-adic (pro-)\'etale Borel--Moore homology}
 Let $X$ be an algebraic scheme of dimension $d_X$ over a field $k$ with structure morphism $\pi_X\colon X\to \Spec k$.
 Let $\ell$ be a prime invertible in $k$.
 We denote the Borel--Moore homology of $X$ with values in $\Z_\ell$ by
$$
H^{i}_{BM}(X, \Z_\ell(n))\coloneq H_{2d_X-i}^{BM}(X, \Z_\ell(d_X-n))= H^{i-2d_X} (X_{\proet},\pi_X^!\widehat \Z_\ell(n-d_X))  ,
$$ 
see \cite{BS} and \cite[Section 4 and Proposition 6.6]{Sch-refined}.
The analogue with $\Z/\ell^r$-coefficients may directly be defined on the \'etale site of $X$ as follows:
$$
H^{i}_{BM}(X, \Z/\ell^r(n))\coloneq H^{i}_{BM}(X, \mu_{\ell^r}^{\otimes n})\coloneq H_{2d_X-i}^{BM}(X, \mu_{\ell^r}^{\otimes d_X-n})=H^{i-2d_X} (X_{\et},\pi_X^!\mu_{\ell^r}^{\otimes n-d_X}) . 
$$ 
We further define
$$
H^{i}_{BM}(X, \Q_\ell(n))\coloneq H^{i}_{BM}(X, \Z_\ell(n))\otimes_{\Z_\ell} \Q_\ell\quad \text{and}\quad H^{i}_{BM}(X, \Q_\ell/\Z_\ell(n))\coloneq \lim_{\substack{\longrightarrow \\ r}}H^{i}_{BM}(X, \Z/\ell^r(n)) .
$$ 
The above defined groups are contravariantly functorial for \'etale maps of schemes of the same dimension and hence in particular for open immersions $U\hookrightarrow X$ with $\dim U=\dim X$, see e.g.\ \cite[Proposition 6.6]{Sch-refined}.

Let $A\in \{\Z/\ell^r,\Z_\ell,\Q_\ell,\Q_\ell/\Z_\ell\}$.
If $Z\subset X$ is closed of codimension $c=\dim X-\dim Z$ and with complement $U=X\setminus Z$, then we have a Gysin (or localization) sequence
\begin{align} \label{les:Gysin}
\dots \to H^{i-2c}_{BM}(Z,A(n-c))\to  H^{i}_{BM}(X, A(n))\to H^{i}_{BM}(U, A(n))\to H^{i+1-2c}_{BM}(Z,A(n-c))\to \dots ,
\end{align}
see e.g.\ \cite[\S 4, (P2) and Proposition 6.6]{Sch-refined}.

If $X$ is smooth and equi-dimensional, then we have $\pi_X^! \cong \pi_X^\ast(d_X)[2d_X]$, by Poincar\'e duality, and so
\begin{align} \label{eq:H_BM=H_proet}
H^{i}_{BM}(X, A(n))=H^{i} (X, A(n))\coloneq H^{i} (X_{\proet},A(n)), 
\end{align}
for all $A\in \{\Z/\ell^r,\Z_\ell,\Q_\ell,\Q_\ell/\Z_\ell\}$, see e.g.\ \cite[Lemma 6.5]{Sch-refined}.
 
\begin{remark}
{\rm
The groups $H^i_{BM}(X,A(n))$  agree up to a reindexing with Borel--Moore homology.
 In this paper we use the cohomological indexing convention, as it makes several of the arguments and statements in this paper, such as Theorem \ref{thm:filtration-intro-BM} or Proposition \ref{prop:weights}, independent of the dimension of $X$.
Since $H^i_{BM}(X,A(n))$ agrees with ordinary cohomology if $X$ is smooth and equi-dimensional, this convention may also make various generalizations from the case of smooth varieties to arbitrary equi-dimensional quasi-projective schemes more transparent.
}
\end{remark}

\subsection{Duality and base change}
 In this subsection we recall some consequences of the six-functor formalism as developed in \cite[Section 6.7]{BS} (see also \cite{Eke} for the analogous results for Jannsen's continuous \'etale cohomology).
The results are well-known and we give some details for convenience of the reader.

To begin with, we denote as usual the compactly supported $\ell$-adic pro-\'etale cohomology of a qcqs $\Z[1/\ell]$-scheme $X$ by
$$
H^i_c(X,\Z_\ell(n))\coloneq H^{i}(\Spec k,R\pi_{X !} \widehat \Z_\ell(n))\quad \text{and}\quad H^i_c(X,\Q_\ell(n))\coloneq H^i_c(X,\Z_\ell(n))\otimes \Q_\ell .
$$ 
\begin{lemma} \label{lem:H_BM-H_c}
Let $X$ be an algebraic scheme of dimension $d_X$ over a separably closed field $k$ and let $\ell$ be a prime invertible in $k$.
Then there is a canonical isomorphism
$$
H^i_{BM}(X,\Q_\ell(n))\cong \Hom_{\Q_\ell}( H^{2d_X-i}_c(X,\Q_\ell(d_X-n)),\Q_\ell)=H^{2d_X-i}_c(X,\Q_\ell(d_X-n))^\ast .
$$
\end{lemma}
\begin{proof}   
In \cite[Section 6.7]{BS}, Bhatt and Scholze prove the six functor formalism for constructible complexes of $\widehat \Z_\ell$-modules on qcqs schemes (see \cite[Definition 6.5.1]{BS}) on which $\ell$ is invertible.
We apply this to a morphism $f\colon X\to Y$ of algebraic schemes over $k$, where $k$ is a separably closed field such that $\ell$ is invertible in $k$. 
The formalism then yields the natural identity
\begin{align} \label{eq:duality}
Rf_\ast \mathcal{RH}om_X(K,f^!L)\stackrel{\cong}\longrightarrow \mathcal{RH}om_Y(Rf_!K,L).
\end{align}
Indeed, for every $M\in D_{\cons}(Y_{\proet},\widehat{\Z}_\ell)$ we have the following natural identities, where we write $\Hom_X\coloneq \Hom_{D_{\cons}(X_\proet,\widehat \Z_\ell)}$ and $\Hom_Y\coloneq \Hom_{D_{\cons}(Y_\proet,\widehat \Z_\ell)}$:
\begin{align*}
\mathrm{Hom}_Y\!\left(M, Rf_\ast \mathcal{RH}om_X(K,f^!L)\right) &\cong \mathrm{Hom}_X\!\left(f^\ast_{comp}  M, \mathcal{RH}om_X(K,f^!L)\right)\\
&\cong \mathrm{Hom}_X\!\left(f^\ast_{comp}  M\widehat \otimes K, f^!L\right) \\
&\cong \mathrm{Hom}_Y\!\left(Rf_!(f^\ast_{comp}  M\widehat \otimes K), L\right) \\
&\cong \mathrm{Hom}_Y\!\left(M\widehat \otimes Rf_!K, L\right) \\
&\cong \mathrm{Hom}_Y\!\left(M, \mathcal{RH}om_Y(Rf_!K,L)\right).
\end{align*}
Here the first and third isomorphisms use the adjunctions
$f^\ast_{comp} \dashv Rf_\ast$ and $Rf_! \dashv f^!$ (see \cite[Lemmas 6.7.2 and 6.7.19]{BS}), the second and last use the tensor--Hom adjunction (see \cite[Lemmas 3.4.11,  6.7.12 and 6.7.13]{BS}), and the fourth is the projection formula (see \cite[Lemma 6.7.14]{BS}). 
(We note that in the above computation, the completed tensor product could be replaced by the ordinary tensor product by \cite[Lemma 6.5.5]{BS}.)
The claim now follows from the Yoneda Lemma. 

We apply \eqref{eq:duality} to the structure morphism $f\coloneq \pi_X\colon X\to \Spec k$ and the locally constant pro-\'etale sheaves $K=\widehat \Z_\ell$ and $L=\widehat \Z_\ell(n-d_X)[-2d_X]$.
We then get
\begin{align*}
H^i_{BM}(X,\Q_\ell(n))\stackrel{\cong}\longrightarrow \mathcal{R}^{i-2d_X}\mathcal{H}om_{\Spec k}(Rf_!\widehat \Z_\ell,\widehat \Z_\ell(n-d_X))\otimes \Q_\ell .
\end{align*}
Moreover,
\begin{align*} 
\mathcal{R}^{i-2d_X}\mathcal{H}om_{\Spec k}(Rf_!\widehat \Z_\ell,\widehat \Z_\ell(n-d_X))\otimes \Q_\ell
&=\mathcal{R}^{i-2d_X}\mathcal{H}om_{\Spec k}(Rf_!\widehat \Z_\ell(d_X-n),\widehat \Z_\ell)\otimes \Q_\ell \\
&=  \Hom_{\Q_\ell}(R^{2d_X-i}f_!\widehat \Z_\ell(d_X-n) \otimes \Q_\ell,  \Q_\ell  )\\ 
&=\Hom_{\Q_\ell}(H_c^{2d_X-i}(X, \Q_\ell(d_X-n)),  \Q_\ell  ).
\end{align*}
This completes the proof.
\end{proof}  
 
\begin{lemma}\label{lem:BM-base-change}
Let
\[
\xymatrix{
X'\ar[r]^-{g'} \ar[d]_{f'} & X \ar[d]^f \\
T \ar[r]^-g & S
}
\]
be a Cartesian diagram of qcqs schemes on which $\ell$ is invertible, with $f$ separated of finite presentation.  
Let $A\coloneq \widehat {\Z}_\ell$ (resp.\ $A=\Z/\ell^r$) and denote the corresponding locally constant pro-\'etale (resp.~\'etale) sheaf on $S$ and $T$ by $A_S$ and $A_T$, respectively.
Then for every $n\in\Z$ there is a natural isomorphism in
$D_{\cons}(T_{\proet},\widehat{\Z}_\ell)$ (resp.~in $D_{\cons}(T_{\et}, \Z/\ell^r)$)
\[
g^* Rf_* f^!A_S(n) \cong Rf'_{*} {f'}^!A_T(n).
\] 
\end{lemma}
\begin{proof}
We treat only the case $A=\widehat \Z_\ell$; the case $A=\Z/\ell^r$ follows via the same argument.  

Let $A\coloneq \widehat \Z_\ell$.
For any $M,N\in D_{\cons}(S_{\proet},\widehat{\Z}_\ell)$ we have $f^!(M\widehat \otimes N)=(f^!M)\widehat \otimes f^*_{comp}N$, which follows from the projection formula \cite[Lemma 6.7.14]{BS} and adjunction.
Hence, $(f^!A_S)(n)=f^!(A_S(n))$ and we get
\begin{align*} %\label{eq:lem:BM-base-change:1}
\mathcal{RH}om_X(A_X(-n),f^!A_S)=\mathcal{RH}om_X(A_X,f^!A_S(n)) =f^!A_S(n).
\end{align*}
Therefore, applying \eqref{eq:duality} to $K=A_X(n)$ and $L=A_S$ yields
\[
Rf_* f^!A_S(n) \cong \mathcal{RH}om_S(Rf_!A_X(-n),A_S).
\]
Pulling back along $g$ and using compatibility of $g^*$ with internal Hom, we obtain
\[
g^*Rf_* f^!A_S(n) \cong \mathcal{RH}om_T(g^*Rf_!A_X(-n),A_T).
\]
By proper base change for $Rf_!$ (see \cite[Lemma 6.7.10]{BS}),
\[
g^*Rf_!A_X(-n)\cong Rf'_{!}A_{X'}(-n),
\]
and therefore
\[
g^*Rf_* f^!A_S(n) \cong \mathcal{RH}om_T(Rf'_{!}A_{X'}(-n),A_T).
\]
Applying  \eqref{eq:duality} again, now to ${f'}$, gives
\[
\mathcal{RH}om_T(Rf'_{!}A_{X'}(-n),A_T)\cong Rf'_{*} \mathcal{RH}om_{X'}(A_{X'}(-n),{f'}^!A_{T}) \cong Rf'_{*}{f'}^!A_T(n),
\]
which proves the claim.
\end{proof}

\begin{proposition} \label{prop:BM-base-change}
Let $\ell$ be a prime and let $A\in \{\Z_\ell,\Q_\ell,\Q_\ell/\Z_\ell,\Z/\ell^r\}$.
Let $f\colon X\to S$ be a separated morphism of finite presentation between qcqs schemes on which $\ell$ is invertible.
Assume that $S$ is Noetherian and integral with function field $k_0$ and geometric generic point $\bar \eta=\Spec k$, where $k/k_0$ is a separable closure.
Then there is a dense open subset $U\subset S$ such that for each geometric point $\bar s\to U\subset S$ and each specialization $\bar \eta\leadsto \bar s$,
there is a natural isomorphism
$$
H^i_{BM}(X_{\bar s},A(n)) \stackrel{\cong}\longrightarrow H^i_{BM}(X_{\bar \eta},A(n)) 
$$
which is compatible with Galois actions in the following sense:
Let $s\in S$ denote the image of $\bar s$ with decomposition group $G_s\subset G_{k_0}=\Gal(k/k_0)$  at $s$.
Then the inertia group $I_s\subset G_s$ at $s$ acts trivially on $H^i_{BM}(X_{\bar \eta},A(n))$ and the induced action of the absolute Galois group $G_{\kappa(s)}\cong G_s/I_s$ of $\kappa(s)$ agrees via the above isomorphism with the natural action on $H^i_{BM}(X_{\bar s},A(n))$. 
\end{proposition} 
\begin{proof}
Up to shrinking $S$, we can assume that $f$ is flat of relative dimension $d$.
Consider the constructible complex $Rf_\ast f^!\widehat{\Z}_\ell(n-d)\in D_{\cons}(S_\proet,\widehat \Z_\ell)$.
By \cite[Proposition 6.6.11]{BS}, we can, up to shrinking $S$ further, assume that this complex is locally constant with perfect values, i.e.~locally isomorphic to $\widehat {\underline{L}}\cong \underline{L}\otimes_{\Z_\ell} \widehat{\Z}_\ell$ for some perfect complex $L$ of $\Z_\ell$-modules.
We claim that under this assumption, the base change assertions claimed in the proposition hold true.

By Lemma \ref{lem:BM-base-change}, and because taking stalks of complexes commutes with taking cohomology, we have  canonical isomorphisms
\begin{align} \label{eq:cor:BM-base-change:1}
H^i_{BM}(X_{\bar s},\widehat{\Z}_\ell(n)) \stackrel{\cong}\longrightarrow \left(R^{i-2d}f_\ast f^!\widehat{\Z}_\ell(n-d) \right)_{\bar s} 
\end{align}
and
\begin{align} \label{eq:cor:BM-base-change:2}
H^i_{BM}(X_{\bar \eta},\widehat{\Z}_\ell(n)) \stackrel{\cong}\longrightarrow \left(R^{i-2d}f_\ast f^!\widehat{\Z}_\ell(n-d) \right)_{\bar \eta} .
\end{align}

The choice of a specialization $\bar \eta\leadsto \bar s$ corresponds to the choice of an $S$-morphism $\bar \eta\to T\coloneq \Spec \mathcal O^{sh}_{S,\bar s}$ to the spectrum of the strict Henselization of $S$ at $\bar s$.
Since $Rf_\ast f^!\widehat{\Z}_\ell(n-d)$ is locally constant by assumption, its restriction to $T$ is constant by \cite[Corollary 6.5.7]{BS}. 
The choice of $T$ therefore induces, in view of \eqref{eq:cor:BM-base-change:1} and \eqref{eq:cor:BM-base-change:2}, a canonical identification
$$
H^i_{BM}(X_{\bar s},\widehat{\Z}_\ell(n)) \stackrel{\cong}\longrightarrow H^i_{BM}(X_{\bar \eta},\widehat{\Z}_\ell(n)) 
$$
which is compatible with the respective Galois actions, as claimed in the proposition.
This proves the proposition for $A=\Z_\ell$ and we obtain the result for $A=\Q_\ell$ by applying $\otimes_{\Z_\ell} \Q_\ell$.

Let us deal with the case $A=\Z/\ell^r$ next.
For each $r$, there is a distinguished triangle
\begin{align*} %\label{eq:cor:BM-base-change-triangle}
Rf_\ast f^!\widehat{\Z}_\ell(n-d)\stackrel{\cdot \ell^r}\longrightarrow Rf_\ast f^!\widehat{\Z}_\ell(n-d)\longrightarrow Rf_\ast f^!\Z/\ell^r(n-d) \stackrel{+1}\longrightarrow .
\end{align*}
Since $Rf_\ast f^!\widehat{\Z}_\ell(n-d)\in D_{\cons}(S_\proet,\widehat \Z_\ell)$ is locally constant with perfect values, the same holds true for $ Rf_\ast f^!\Z/\ell^r(n-d) $. 
The same argument as above then yields an isomorphism
$$
H^i_{BM}(X_{\bar s},\Z/\ell^r(n)) \stackrel{\cong}\longrightarrow H^i_{BM}(X_{\bar \eta},\Z/\ell^r(n)) ,
$$
that is compatible with Galois actions.
Taking direct limits, we obtain the same result for $A=\Q_\ell/\Z_\ell$. 
(This step uses the fact that we do not have to shrink $S$ further, once we have ensured that $Rf_\ast f^!\widehat{\Z}_\ell(n-d)$ is locally constant; in particular, our choice of $S$ does not depend on $r$.)
This concludes the proof. 
\end{proof}

\subsection{Refined unramified $\ell$-adic cohomology}

Let $X$ be an algebraic scheme over a field $k$.
We let $F_jX$ be the pro-scheme that consists of all open subsets $U\subset X$ whose complement $Z=X\setminus U$ has codimension $\dim X-\dim Z \geq j+1$.
We then define 
\begin{align} \label{eq:H_BM(FjX)}
H^i_{BM}(F_jX,A(n))\coloneq \lim_{\substack{\longrightarrow \\ U\subset X}}H^i_{BM}(U,A(n)) ,
\end{align}
where $U\subset X$ runs through all open subsets with $F_jX\subset U$, i.e.\ all open subsets that belong to the pro-scheme $F_jX$.
The $j$-th refined unramified cohomology of $X$ is then defined by
\begin{align} \label{eq:H_j,nr(X)}
H^i_{j,\nr}(X,A(n))\coloneq \im(H^i_{BM}(F_{j+1}X,A(n))\to H^i_{BM}(F_{j}X,A(n))) .
\end{align}

Let $k_0\subset k$ be a subfield such that $k/k_0$ is Galois with group $G$ and such that $X=X_0\times_{k_0}k$.
If $U\subset X$ is open with $F_jX\subset U$, then we can replace the complement $Z\coloneq X\setminus U$ by its Galois orbit to construct an open subset $U'\subset X$ which is defined over $k_0$ and satisfies $U'\subset U$ and $F_jX\subset U'$.
This shows that  in the limit \eqref{eq:H_BM(FjX)} we may  run only over those open subsets that are defined over $k_0$.
Hence, $H^i_{BM}(F_jX,A(n))$ and $H^i_{j,\nr}(X,A(n))$ carry natural $A$-linear $G$-actions and hence are $A$-$G$-modules.

\subsection{Motivic cohomology} \label{subsec:motivic-cohomology}
Let $X$ be a smooth equi-dimensional scheme over a field $k$ and let $A$ be an abelian group. 
We consider Bloch's cycle complex
\begin{align} \label{def:A(n)_zar}
A(n)\coloneq A(n)_\zar\coloneq z^n(-_{\zar},\bullet)[-2n]\otimes^{\mathbb L} A  \in D(X_\zar) 
\end{align}
with values in $A$, which is a complex of sheaves in the Zariski topology of $X$, see \cite{bloch-motivic}.
  By convention, this complex is zero for $n<0$.
We denote the $i$-th cohomology sheaf of $A(n)_\zar$ by $\mathcal H^i_M(A(n))$ and define the motivic cohomology of $X$ with values in $A$ by
$$
H^i_M(X,A(n))\coloneq  H^i(X_\zar, A(n)_\zar) .
$$
By \cite[page 269, (iv)]{bloch-motivic}, motivic cohomology with integral coefficients is canonically isomorphic to higher Chow groups:
\begin{align} \label{eq:motivic=higher-Chow}
H^i_M(X,\Z(n))\cong \CH ^n(X,2n-i) .
\end{align}
In particular,  $H^i_M(X,\Z(n))$ agrees with Voevodsky's definition of motivic cohomology, see \cite{voevodsky-imrn}.

We list two more consequences of \eqref{eq:motivic=higher-Chow}.
Firstly,
\begin{align} \label{eq:Hi_M(X,n)=0-for-i}
H^i_M(X,\Z(n))=0\quad \text{for $i>\dim X+n$}
\end{align}
because classes in $\CH ^n(X,2n-i)$ are represented by codimension $n$ cycles on $X\times \Delta^{2n-i}$.
Secondly, if $X$ admits a model $X_0$ over a subfield $k_0\subset k$, then
\begin{align} \label{eq:limit-HiM}
H^i_M(X,\Z(n))\cong \lim_{\substack{\rightarrow\\
L/k_0}}H^i_M(X_0\times_{k_0}L,\Z(n))
\end{align}
where the direct limit runs through all subfields $L\subset k$ which contain $k_0$ and such that $L/k_0$ is finitely generated.

\begin{lemma} \label{lem:curlyH^iZ_motivic(n)=0}
Let $X$ be a smooth variety over a field $k$ and let $m$ be a positive integer invertible in $k$.
Then, for $i>n$, we have $\mathcal H^i_M(\Z(n))=0$ and $\mathcal H^i_M((\Z/m)(n))=0$.
\end{lemma} 
\begin{proof}
The vanishing of $\mathcal H^i_M(\Z(n))$ for $i>n$ follows from the Gersten Conjecture for higher Chow groups \cite[Theorem 10.1]{bloch-motivic} and the fact that $\CH^n(\Spec \kappa(x), 2n-i)=0$ for all $i>n$ and all $x\in X$.
The vanishing of  $\mathcal H^i_M((\Z/m)(n))$ follows from this via the long exact Bockstein sequence $$
\dots \longrightarrow \mathcal H^i_M(\Z(n))\stackrel{\times m}\longrightarrow \mathcal H^i_M(\Z(n))\longrightarrow \mathcal H^i_M(\Z/m(n))\longrightarrow \mathcal H^{i+1}_M(\Z(n))\longrightarrow \dots\ .
$$
This proves the lemma. 
\end{proof}

The pullback of $A(n)_\zar$ to the small \'etale site of $X$ is denoted by $A(n)_\et$.
If $m\geq 2$ denotes an integer invertible in $k$, then, by \cite[Theorem 1.5]{geisser-levine}, we have 
\begin{align} \label{eq:geisser-levine}
\Z/m(n)_\et\cong \mu_{m}^{\otimes n} .
\end{align}

\begin{theorem}  \label{thm:prelim-comparison-H_m-to-H_et}
Let $X$ be a smooth algebraic scheme over a field $k$ and let $\ell$ be invertible in $k$.
Let $\pi\colon X_\et\to X_\zar$ be the natural map of sites. 
Then, for $j\leq n$, the natural map
$$
\tau_{\leq j} \Z/\ell^r(n)_\zar \longrightarrow \tau_{\leq j}\R\pi_\ast \Z/\ell^r(n)_\et 
$$
is an isomorphism. 
\end{theorem} 
\begin{proof}
By the work of Suslin--Voevodsky \cite{suslin-voevodsky} and Geisser--Levine \cite[Theorem 1.5]{geisser-levine}, this is a consequence of the Beilinson--Lichtenbaum conjecture proven by Rost and Voevodsky \cite{Voevodsky}, see also \cite[Theorem 6.6]{Voe-milnor}. 
\end{proof}

\section{Mixed and ind-mixed Galois modules} \label{sec:mixed-and-ind-mixed}

\subsection{Mixed Galois modules}
Let $k_0$ be a finitely generated field with separable closure $k=k_0^{\sep}$.
Let $G_{k_0}=\Gal(k/k_0)$ be the absolute  Galois group of $k_0$.
Pick a normal finite type $\Z$-scheme $S$ with function field $k_0$.
For a closed point $s\in S$ we have the decomposition group $G_s\subset G_{k_0}$ and a surjection $\varphi_s:G_s\twoheadrightarrow G_{\kappa(s)}$ with inertia group $I_s=\ker(\varphi_s)$.
 
We say that a continuous $\Q_\ell$-$G_{k_0}$-module $M$ is pure of weight $w$ if it is finite-dimensional as $\Q_\ell$-vector space,  
and the following holds up to shrinking $S$ (and possibly replacing $S$ by a purely inseparable dominant cover): 
for all closed points $s\in S$, 
the inertia group $I_s$ acts trivially on $M$ and the induced $G_{\kappa(s)}$-action has the property that the geometric Frobenius element in $G_{\kappa(s)}$ acts with eigenvalues of absolute value $q_s^{w/2}$ (where $\kappa(s)=\F_{q_s}$) with respect to any embedding $\Q_\ell\hookrightarrow \C$. 

More generally, a $\Q_\ell$-$G_{k_0}$-module $M$ is mixed if it is finite-dimensional as a $\Q_\ell$-vector space and if it has a descending weight filtration $W_\ast$ whose graded quotients are pure $\Q_\ell$-$G_{k_0}$-modules of some weight. 
We say that the weights $w$ of $M$ are contained in a subset $B\subset \Z$ if the weights of the nonzero graded quotients of $M$ are contained in $B$.

We recall the following simple lemma for the comfort of the reader.

\begin{lemma} \label{lem:endo-eigenvalues-sequence}
In the above notation, let  $M_{1} \to M_{2} \to M_{3}$ be an exact sequence of mixed $\Q_\ell$-$G_{k_0}$-modules.
Then the weights of $M_2$ are contained in the union of the weights of $M_1$ and $M_3$. 
\end{lemma}
\begin{proof}
This follows from the elementary fact that for an exact sequence of finite-dimensional vector spaces $V_{1} \to V_{2} \to V_{3}$ over a field and an endomorphism $\phi$ on this sequence,
the eigenvalues of $\phi$ on $V_{2}$ are contained in the union of the eigenvalues of $\phi$ on $V_{1}$ and those of $\phi$  on $V_{3}$. 
\end{proof}

The following result provides an important source of mixed Galois modules,  cf.\ \cite[Lemma 3.11]{balkan-Sch}. 

\begin{proposition} \label{prop:weights}
Let $k_0$ be a finitely generated field, let $k$ be the separable closure of $k_0$ and let $G=\Gal(k/k_0)$.
Let $Z_0$ be an equi-dimensional algebraic scheme over $k_0$ with base change $Z\coloneq Z_0\times_{k_0}k$.
Then the $\Q_\ell$-$G$-module $H^i_{BM}(Z,\Q_\ell(n))$ is mixed with weights contained in $\{i-2n,i-2n+1,\dots, i-2n+i\}$.
\end{proposition}
\begin{proof}  
Let $S=\Spec A$ for a finite type $\Z$-algebra $A$ with fraction field $k_0$, such that there is a separated scheme of finite type $f\colon \mathcal Z\to S$ whose generic fibre is isomorphic to $Z$.
Then the complex $Rf_\ast f^!\widehat{\Q_\ell}(n)$ is constructible.
By the base change result in Lemma \ref{lem:BM-base-change}, it follows that we may, by \cite[Proposition 6.6.11]{BS} and up to shrinking $S$, assume that $Rf_\ast f^!\widehat{\Q_\ell}(n)$ is locally constant.
By a result of Deligne \cite{deligneII}, the cohomology sheaves of $Rf_\ast f^!\widehat{\Q_\ell}(n)$ are pointwise mixed, see \cite[Proposition 3.2]{Huber} and \cite[\S 2.1, \S 2.6]{morel}.
By the base change result in Lemma \ref{lem:BM-base-change} this implies  that the $\Q_\ell$-$G$-module $H^i_{BM}(Z,\Q_\ell(n))$ is mixed, cf.~Proposition \ref{prop:BM-base-change}. 

It remains to compute the weights of $H^i_{BM}(Z,\Q_\ell(n))$.
Up to replacing $Z$ by its reduction, we can, by the topological invariance of the pro-\'etale site (see \cite[Lemma 5.4.2]{BS}), assume that it is reduced.
By a similar argument we can, up to replacing $k_0$ by a purely inseparable extension, assume that $Z$ is geometrically reduced.
In particular, $Z_0$ is generically smooth.
We pick a smooth affine open subset $U_0\subset Z_0$ such that the complement $W_0\subset Z_0$ is of pure codimension one.
(This can always be done by replacing components of the wrong codimension by the closure of a suitable hyperplane section in an affine chart that contains the generic point of that component.)
We denote by $U=U_0\times_{k_0}k$ and $W=W_0\times_{k_0}k$ the base changes of $U_0$ and $W_0$.
Then $U\subset Z$ is a smooth dense open subset which is stable under the $G$-action, such that $W=Z\setminus U$ is equi-dimensional of codimension one in $Z$.
By \eqref{les:Gysin}, there is an exact sequence
of finite dimensional $\Q_\ell$-vector spaces with a Galois action 
$$
H^{i-2}_{BM}(W,\Q_\ell(n-1))\longrightarrow 
H^i_{BM}(Z,\Q_\ell(n))\longrightarrow H^i_{BM}(U,\Q_\ell(n)) .
$$
By induction on the dimension of $Z$, we may assume that $H^{i-2}_{BM}(W,\Q_\ell(n-1))$ is mixed with weights contained in
$ 
\{i-2n,i-2n+1,\dots ,2i-2n-2\}
$.
In order to show that $H^i_{BM}(Z,\Q_\ell(n))$ is mixed with weights contained in $\{i-2n,i-2n+1,\dots ,2i-2n\}$, it thus suffices by Lemma \ref{lem:endo-eigenvalues-sequence} to prove the same for $H^i_{BM}(U,\Q_\ell(n))$.
Since $U$ is smooth and equi-dimensional, we have $H^i_{BM}(U,\Q_\ell(n))\cong H^i(U_\proet,\Q_\ell(n))$, see \eqref{eq:H_BM=H_proet}.
Moreover, since $k$ is separably closed, $H^i(U_\proet,\Q_\ell(n))\cong H^i(U_\et,\Q_\ell(n))$ (see \eqref{eq:H_proet=H_et}) and so it suffices to deal with ordinary $\ell$-adic \'etale cohomology of $U$.
This case is well-known; we include some details for convenience of the reader.
Recall first that for a proper generically finite morphism $f$ between smooth varieties, we have $f_\ast f^\ast=\deg(f)\cdot \id$, see e.g.~\cite[Lemma A.11]{Sch-moving}.
Via de Jong's alterations \cite[Theorem 4.1 and Remark 4.2]{deJong}, the problem can therefore be reduced to the case where $U$ admits a smooth projective compactification $U\subset Y$ whose complement $E=Y\setminus U$ is a simple normal crossing divisor.
(More precisely, to obtain a smooth alteration and not just a regular one, we apply de Jong's theorem to the base change of $U$ to the perfect closure of $k$ and then descend the result to a finite purely inseparable extension of $k_0$---in this last step we have to replace the scheme $S$ above by a purely inseparable finite cover and $k_0$ by a purely inseparable finite field extension.)
Let $E_i$ with $i\in I$ be the irreducible components of $E$.
For a subset $J\subset I$ we further put $E_J\coloneq \bigcap_{j\in J}E_j$ and let
$$
E^{[q]}=\bigcup_{J\subset I, |J|=q}E_J .
$$
We have a convergent spectral sequence
$$
E_2^{p,q}=H^p(E^{[q]}_\et,\Q_\ell(n-q))\Longrightarrow H^{p+q}(U_\et,\Q_\ell(n)) ,
$$
see e.g.\ \cite[(2.4)]{jannsen-weights}.
By \cite{deligne}, $E_2^{p,q}$ is pure of weight $p-2n+2q$.
Moreover, we have $E_2^{p,q}=0$ for $p<0$ or $q<0$ and so only the terms $E_2^{i-q,q}$ for $q=0,\dots ,i$ contribute to $H^i(U_\et,\Q_\ell(n))$.
This implies that the weights of $H^i(U_\et,\Q_\ell(n))$ are contained in 
$$
\{i-2n+q \mid q=0,\dots ,i\}=\{i-2n,i-2n+1,\dots ,i-2n+i\}.
$$
This concludes the proof of the proposition.
\end{proof}

\subsection{Ind-mixed Galois modules}
By definition, a mixed $\Q_\ell$-$G_{k_0}$-module is finite-dimensional as a $\Q_\ell$-vector space.
It is convenient to have the following generalization to infinite-dimensional vector spaces.

\begin{definition} \label{def:ind-mixed}
A $\Q_\ell$-$G_{k_0}$-module $M$ is ind-mixed with weights contained in a subset $B\subset \Z$, if 
$$
M\cong \lim_{\substack{\longrightarrow\\ i\in I}} M_i
$$ 
is isomorphic to a direct limit of $\Q_\ell$-$G_{k_0}$-modules $M_i$, such that the following holds for all $i\in I$:
\begin{enumerate}
\item $M_i$ is finite-dimensional  as $\Q_\ell$-vector space;
\item $M_i$ is mixed with weights contained in $B$.
\end{enumerate} 
\end{definition}

In view of Proposition \ref{prop:weights}, we have the following example.

\begin{example} \label{ex:H^i(F_jX)-H_j,nr-weights}
Let $X$ be an algebraic $k$-scheme with model over $k_0$.
Then the group $H^i_{BM}(F_jX,\Q_\ell(n))$ from \eqref{eq:H_BM(FjX)} and the refined unramified $\ell$-adic cohomology group $H^i_{j,\nr}(X,\Q_\ell(n))$ from \eqref{eq:H_j,nr(X)} are ind-mixed $\Q_\ell$-$G_{k_0}$-modules.
\end{example}

For $X$ smooth and projective, we shall compute the weights of $H^i_{j,\nr}(X,\Q_\ell(n))$ in Proposition \ref{prop:weights-H_j,nr} below.
To this end, we will need a couple of general results on ind-mixed Galois modules that we collect next.

\begin{lemma} \label{lem:ind-mixed}
A $\Q_\ell$-$G_{k_0}$-module $M$ is ind-mixed (with weights contained in $B\subset \Z$) if and only if $M$ is the union of its finite-dimensional  $\Q_\ell$-$G_{k_0}$-submodules and each of these submodules is mixed (of weights contained in $B$).
\end{lemma}
\begin{proof}
If $M$ is the union of its finite-dimensional  $\Q_\ell$-$G_{k_0}$-submodules and each of these submodules is mixed of weights contained in $B$, then clearly $M$ is ind-mixed with weights contained in $B$.
For the converse, assume that $M$ is ind-mixed.
In the notation of Definition \ref{def:ind-mixed}, we can replace each $M_i$ by its image in $M$, which is still a mixed $\Q_\ell$-$G_{k_0}$-module with weights contained in $B$.
But then the direct limit in Definition \ref{def:ind-mixed} turns into a union of mixed modules and we see that an ind-mixed $\Q_\ell$-$G_{k_0}$-module with weights contained in $B$ is a union of mixed $\Q_\ell$-$G_{k_0}$-modules with weights contained in $B$. 
It also shows that any finite-dimensional $\Q_\ell$-$G_{k_0}$-submodule of  $M$ is contained in some mixed  $\Q_\ell$-$G_{k_0}$-submodule with weights contained in $B$ and hence it is itself mixed with weights contained in $B$.
This proves the lemma. 
\end{proof} 

A morphism of ind-mixed $\Q_\ell$-$G_{k_0}$-modules is a morphism of the underlying $\Q_\ell$-$G_{k_0}$-modules.

Lemma \ref{lem:ind-mixed} implies that images and kernels of morphisms of ind-mixed $\Q_\ell$-$G_{k_0}$-modules are again ind-mixed $\Q_\ell$-$G_{k_0}$-modules. 
More generally, one easily sees that the category of ind-mixed $\Q_\ell$-$G_{k_0}$-modules is abelian; in fact, this category is nothing but the ind-completion of the (Tannakian) category of mixed $\Q_\ell$-$G_{k_0}$-modules.

We will also need the following simple lemma.

\begin{lemma} \label{lem:im(f:MtoN)}
Let $f\colon M\to N$ be a morphism of ind-mixed $\Q_\ell$-$G_{k_0}$-modules.
If $M$ has weights contained in $B\subset \Z$, then the same holds for the image $f(M)\subset N$.
\end{lemma}
\begin{proof}
If 
$$
M\cong \lim_{\substack{\longrightarrow \\ i\in I}} M_i\quad \text{then}\quad f(M)\cong \lim_{\substack{\longrightarrow \\ i\in I}} f(M_i)
$$
and so the lemma follows from the fact that the image of a mixed $\Q_\ell$-$G_{k_0}$-module with weights contained in $B$ is again mixed with weights contained in $B$. 
\end{proof}

We conclude this section with the following two results which compute the weights of certain natural ind-mixed $\Q_\ell$-$G$-modules.

\begin{proposition} \label{prop:weights-H_j,nr}
Let $k_0$ be a finitely generated field and let $k$ be the separable closure of $k_0$ with Galois group $G=\Gal(k/k_0)$.
Let $X_0$ be a smooth projective variety over $k_0$ and set $X\coloneq X_0\times_{k_0} k$. 
Then the weights $w$ of the ind-mixed $\Q_\ell$-$G$-module $H^i_{j,nr}(X,\Q_\ell(n))$ satisfy
$$
i-2n\leq w\leq \max(i-2n, 2i-2n-2j-2) .
$$ 
\end{proposition}
\begin{proof}
It suffices to show that any finitely generated $\Q_\ell$-$G$-module $M\subset  H^i_{j,nr}(X,\Q_\ell(n))$ that is finitely generated as a $\Q_{\ell}$-vector space has weights as claimed in the proposition.
There is an open subset $U\subset X$ with $F_{j+1}X\subset U$ such that 
$$
M\subset \im( H^i(U_\et,\Q_\ell(n))\longrightarrow H^i(F_jX,\Q_\ell(n)) ) .
$$ 
Using Lemma \ref{lem:im(f:MtoN)}, it thus suffices to show that the weights of the Galois module $H^i(U_\et,\Q_\ell(n))$ are as claimed.
To show this, let $Z\coloneq X\setminus U$.
Replacing $Z$ by a suitable complete intersection that contains it, we can without loss of generality assume that  $Z$ is pure-dimensional of codimension $j+2$ in $X$.
Then the localization sequence \eqref{les:Gysin} yields a short exact sequence
$$
H^i(X_\et,\Q_\ell(n))\longrightarrow H^i(U_\et,\Q_\ell(n))
\longrightarrow H^{i-2j-3}_{BM}(Z,\Q_\ell(n-j-2)) .
$$
By \cite{deligne}, $H^i(X_\et,\Q_\ell(n))$ is pure of weight $i-2n$.
Moreover, by Proposition \ref{prop:weights} above,  $H^{i-2j-3}_{BM}(Z,\Q_\ell(n-j-2))$ has weights $w$ in the interval 
$$
i-2j-3-2n+2j+4=i-2n+1\leq w \leq i-2n+1+i-2j-3=2i-2j-2n-2 .
$$
This proves the proposition.
\end{proof}

\begin{proposition} \label{prop:weights-H^p(X,curly-H^q)}
Let $k_0$ be a finitely generated field and let $k$ be the separable closure of $k_0$ with Galois group $G=\Gal(k/k_0)$.
Let $X_0$ be a smooth quasi-projective variety over $k_0$ and set $X\coloneq X_0\times_{k_0} k$.  
Then the $\Q_\ell$-$G$-module
$$
H^p(X_\zar, \R^q \pi^{\proet}_\ast \widehat\Q_\ell(n))
$$ 
is ind-mixed and its weights $w$ satisfy
$$
p+q-2n\leq w\leq \max(p+q-2n, 2q-2n) . % \quad \text{if $p\geq 2$}
$$
If $X$ is projective and $p\leq 1$, then the stronger conclusion 
$$
p+q-2n\leq w\leq \max(p+q-2n,2q-2n-2)  
$$ 
holds true.
\end{proposition}
\begin{proof}
Note that the Gersten conjecture holds for $\ell$-adic pro-\'etale cohomology (see e.g.\ \cite{Sch-moving}).
It follows that $H^p(X_\zar, \R^q \pi^{\proet}_\ast \widehat \Q_\ell(n))$ is a subquotient of $\bigoplus_{x\in X^{(p)}} H^{q-p}(x,\Q_\ell(n-p))$, where 
$$
H^{q-p}(x,\Q_\ell(n-p))=H_{BM}^{q-p}(F_0\overline{\{x\}},\Q_\ell(n-p)). 
$$ 
By Example \ref{ex:H^i(F_jX)-H_j,nr-weights}, this is an ind-mixed $\Q_\ell$-$G_{k_0}$-module, whose weights are contained in 
$$
\{q+p-2n,q+p-2n+1,\dots , 2q-2n\} ,
$$
see Proposition \ref{prop:weights}.
Hence, the same holds for $\bigoplus_{x\in X^{(p)}} H^{q-p}(x,\Q_\ell(n-p))$ and hence also for the subquotient $H^p(X,\R^q \pi^{\proet}_\ast \widehat \Q_\ell(n))$.
This proves the first part of the proposition. % in the case where $p\geq 2$.

Let now $X$ be smooth projective.
By \cite[Proposition 7.35]{Sch-refined}, there is an exact sequence 
$$
H^{p+q-1}_{p-2,nr}(X,\Q_\ell(n)) \longrightarrow H^p(X_\zar, \R^q \pi^{\proet}_\ast \widehat\Q_\ell(n))\longrightarrow H^{p+q}_{p,nr}(X,\Q_\ell(n)) .
$$ 
By Proposition \ref{prop:weights-H_j,nr}, the weights $w$ of $H^{p+q}_{p,\nr}(X,\Q_\ell(n))$ satisfy: 
$$
p+q-2n\leq w\leq \max(p+q-2n ,2q-2n-2) .
$$
If $p<2$, then $H^{p+q-1}_{p-2,nr}(X,\Q_\ell(n))=0$ and so the second assertion in the proposition follows.
\end{proof}

\begin{remark}
{\rm
Since $k$ is separably closed,
the  sheaf $\R^q \pi^{\proet}_\ast \widehat \Q_\ell(n)$ agrees by \eqref{eq:H_proet=H_et} with the sheaf $\mathcal H^q_{\et}(\Q_\ell(n))$ associated to the presheaf $U\mapsto H^q(U_\et,\Q_\ell(n))$.
It follows that Proposition \ref{prop:weights-H^p(X,curly-H^q)} applies to $H^p(X,\mathcal H^q_{\et}(\Q_\ell(n)))$.
}
\end{remark}
 
 \subsection{Some consequences over finite fields} \label{sec:vanishing-H_j,nr-finite-field}

\begin{proposition} \label{prop:H_BM-vanishing-finite}
Let $X$ be a quasi-projective scheme over a finite field $k$ and let $\ell$ be a prime invertible in $k$.
Then $H^{i}_{BM}(X,\Z_\ell(n))$ is finite for $i<n$.
\end{proposition}
\begin{proof}
Let $\bar k$ be an algebraic closure of $k$ and let $G=\Gal(\bar k/k)$ be the absolute Galois group of $k$.
Since $k$ is a finite field, $G\cong \widehat \Z$ has cohomological dimension 1.
By the Hochschild--Serre spectral sequence (given by the composed functor spectral sequence), we thus get a short exact sequence
$$
0\longrightarrow H^{i-1}_{BM}(X_{\bar k},\Z_\ell(n))_G\longrightarrow H^{i}_{BM}(X,\Z_\ell(n))\longrightarrow H^{i}_{BM}(X_{\bar k},\Z_\ell(n))^G\longrightarrow 0 .
$$
By Proposition \ref{prop:weights}, $ H^{i-1}_{BM}(X_{\bar k},\Q_\ell(n))$ and $ H^{i}_{BM}(X_{\bar k},\Q_\ell(n))$ are mixed Galois modules of negative weights, because $i<n$.
This implies that $H^{i-1}_{BM}(X_{\bar k},\Z_\ell(n))_G$ and $H^{i}_{BM}(X_{\bar k},\Z_\ell(n))^G$ are finite, hence so is $H^{i}_{BM}(X,\Z_\ell(n))$.
This concludes the proof.
\end{proof}

\begin{proposition} \label{prop:weights-H_j,nr-finite-field}
Let $k$ be a finite field and let $X$ be a smooth projective equi-dimensional scheme over $k$. 
Then $H^i_{j,nr}(X,\Q_\ell(n))=0$ for $i<\min(2n,n+j+1)$. 
\end{proposition}
\begin{proof}
By \eqref{eq:H_BM(FjX)} and \eqref{eq:H_j,nr(X)}, any class in $H^i_{j,nr}(X,\Q_\ell(n))$ is represented by a cohomology class $\alpha\in H^i(U_\et,\Q_\ell(n))$ for some open subset $U\subset X$ with $F_{j+1}X\subset U$.
Let $\bar U=U\times_k \bar k$, where $\bar k$ is the algebraic closure of $k$.
Let $G=\Gal(\bar k/k)$ be the Galois group of $k$.
By the Hochschild--Serre spectral sequence, which degenerates by cohomological dimension reasons, we have an exact sequence
$$
0\longrightarrow H^{i-1}(\bar U_\et,\Q_\ell(n))_G\longrightarrow H^i(U_\et,\Q_\ell(n))\longrightarrow H^i(\bar U_\et,\Q_\ell(n))^G\longrightarrow 0 .
$$
The weights $w$ of $H^i(\bar U_\et,\Q_\ell(n))$ satisfy (see the proof of Proposition \ref{prop:weights-H_j,nr}), 
$$
i-2n\leq w\leq \max(i-2n,2(i-n-j-1)).
$$

Let now $i<\min(2n,n+j+1)$.
Then the above inequality shows that $w<0$ and so $H^i(\bar U_\et,\Q_\ell(n))^G=0$.
Arguing similarly, we see that the weights $w$ of $H^{i-1}(\bar U_\et,\Q_\ell(n))$ satisfy
$$
i-1-2n\leq w\leq \max(i-1-2n,2(i-1-n-j-1))\leq \max(i-2n,2(i-n-j-2))<0.
$$
Hence, $H^{i-1}(\bar U_\et,\Q_\ell(n))_G=0$ and so $H^i(U_\et,\Q_\ell(n))=0$, as we want.
This concludes the proof of the proposition.
\end{proof}
 
\begin{corollary} \label{cor:weights-truncated-proetal-finite-field}
Let $k$ be a finite field and let $X$ be a smooth projective equi-dimensional scheme over $k$. 
Assume $j< n$ and $i< 2n$.  
Then $H^i(X_\zar,\tau_{\leq j} \R\pi^{\proet}_\ast \Q_\ell(n)_{\proet})=0$. 
\end{corollary}
\begin{proof} 
The canonical truncation triangle gives an exact sequence
$$
H^{i-1}(X_\zar,\tau_{\geq j+1} \R\pi^{\proet}_\ast \Q_\ell(n)_{\proet})\longrightarrow H^i(X_\zar,\tau_{\leq j} \R\pi^{\proet}_\ast \Q_\ell(n)_{\proet})\longrightarrow H^i(X_\zar,\R\pi^{\proet}_\ast \Q_\ell(n)_{\proet})
$$
Here,
$$
H^i(X_\zar,\R\pi^{\proet}_\ast \Q_\ell(n)_{\proet})\cong H^i(X_\et,\Q_\ell(n))
$$ 
(see \eqref{eq:H_proet=H_et}) vanishes for $i\notin \{2n,2n+1\}$ by the Hochschild--Serre spectral sequence and Deligne's results on weights,  see \cite[p.\ 780, (28)]{CTSS}. 
By \cite[Theorem 1.2]{AS24}, we have a canonical isomorphism
$$
H^{i-1}(X_\zar,\tau_{\geq j+1} \R\pi^{\proet}_\ast \Q_\ell(n)_{\proet})\cong H^{i-1}_{i-j-2,nr}(X,\Q_\ell(n)) .
$$
It thus suffices to show that
$$
H^{i-1}_{i-j-2,nr}(X,\Q_\ell(n))=0
$$
for $j<n$ and $i< 2n$.
By Proposition \ref{prop:weights-H_j,nr-finite-field}, the above group vanishes if 
$$
i-1<\min(2n,n+i-j-2+1)=\min(2n,n+i-1-j).
$$
Equivalently, we need that $i<2n+1$ and $i<n+i-j$.
The former holds because $i\leq 2n$.
The latter is equivalent to $j<n$, which holds by assumption.
This concludes the proof of the corollary.
\end{proof}

\section{Refined unramified motivic cohomology} 
\label{sec:proof-of-thm:H_M-unramified}

\subsection{Unramified and refined unramified motivic cohomology} \label{subsec:H_M,nr}

In this subsection it is convenient to work with higher Chow groups of singular varieties, i.e.~with Borel--Moore motivic homology.
As in \eqref{eq:H^i-BM,M-intro}, we adopt the convention that for any equi-dimensional quasi-projective $k$-scheme $X$ of dimension $d_X$, we write 
$$
H^i_{BM,M}(X,A(n))\coloneq H_{2d_X-i}^{BM,M}(X,A(d_X-n))\coloneq H^{i}(X_\zar,A^{BM}(n)),
$$
where 
\begin{align} \label{eq:A^BM}
A^{BM}(n)=z^n(-_{\zar},\bullet)[-2n]\in D(X_\zar)
\end{align} 
is Bloch's cycle complex from \eqref{def:A(n)_zar}.\footnote{The superscript BM in the complex $A^{BM}(n)$ is only used to distinguish this complex from Voevodsky's motivic complex; both complexes are quasi-isomorphic if $X$ is smooth and equi-dimensional but not in general.}
If $X$ is smooth and equi-dimensional, then $H^i_{BM,M}(X,A(n))=H^i_{M}(X,A(n))$ by our previous definition.

By \cite[p.\ 269, (iv)]{bloch-motivic} (see also \cite{bloch-JAG}), we have a canonical isomorphism
\begin{align}\label{eq:H_BM,M=Chow}
H^i_{BM,M}(X,A(n))\cong \CH^n(X,A;2n-i)\coloneq H_{2n-i}(z^{n}(X_\zar,\bullet)\otimes_{\Z}^{\mathbb L} A). 
\end{align}
(The derived tensor product on the right could be replaced by an ordinary one, because $z^{n}(X_\zar,\bullet)$ is a complex of sheaves of free $\Z$-modules.)

We list some consequences of \eqref{eq:H_BM,M=Chow} in what follows.
Firstly,
\begin{align}\label{eq:H_BM,M-Ql/Zl}
H^i_{BM,M}(X,\Q_\ell/\Z_\ell(n))\cong \colim_r  H^i_{BM,M}(X,\Z/\ell^r(n)) .
\end{align}
Secondly,   
if $Z\subset X$ is a closed equi-dimensional subscheme  of codimension $c$ with complement $U=X\setminus Z$, then, by \cite{bloch-JAG}, there is a functorial long exact localization sequence
\begin{align}\label{eq:localization-H_BM,M}
H^i_{BM,M}(X,A(n))\longrightarrow H^i_{BM,M}(U,A(n))\stackrel{\del} \longrightarrow H^{i+1-2c}_{BM,M}(Z,A(n-c))\longrightarrow H^{i+1}_{BM,M}(X,A(n)) .
\end{align} 
Thirdly, for a variety $X$ over a field $k$, we have
\begin{align}\label{eq:H_BM,M-fields}
H^i_M(k(X),A(n))\coloneq H^i_M(\Spec k(X),A(n))\cong \lim_{\substack{\longrightarrow\\ U}}  H^i_{BM,M}(U,A(n))
\end{align}
where $U$ runs through all dense open subsets of $X$.
To see the above isomorphism, note first that in the definition of $H^i_M(k(X),A(n))$ in \eqref{eq:H_BM,M-fields}, we view $\Spec k(X)$ as a smooth scheme over $k(X)$; in particular, $H^i_M(k(X),A(n))\cong \CH^n(k(X),A;2n-i)$. 
Hence, the isomorphism in question follows from \eqref{eq:H_BM,M=Chow} and 
$$
\CH^n(k(X),A;2n-i)\cong \lim_{\substack{\longrightarrow\\  U\subset X,\ \text{dense}}}  \CH^n(U,A;2n-i).
$$
We also note that if $X$ is generically smooth, then, after restricting to the
cofinal system of smooth dense open subsets $U\subset X$, the terms
$H^i_{BM,M}(U,A(n))$ in \eqref{eq:H_BM,M-fields} may be replaced by
$H^i_M(U,A(n))$.

Combining \eqref{eq:localization-H_BM,M} and \eqref{eq:H_BM,M-fields}, we obtain, for all $x\in X^{(1)}$, residue maps
$$
\del_x\colon H^i_M(k(X),A(n))\longrightarrow H^{i-1}_M(\kappa(x),A(n-1)) .
$$ 
The unramified motivic cohomology with values in $A$ is then defined by
$$
H^i_{M,nr}(X,A(n))\coloneq\{\alpha\in H^i_M(k(X),A(n))\mid \del_x\alpha=0\quad \forall x\in X^{(1)}\}.
$$

If $X$ is smooth and equi-dimensional, 
equivalent definitions may be given analogously to \cite[Theorem 4.1.1]{CT}.
For instance, by the Gersten conjecture proven by Bloch \cite{bloch-motivic,bloch-JAG}, we have a canonical isomorphism
\begin{align} \label{eq:H_nr=H^0-curlyH}
H^i_{M,nr}(X,A(n))\cong H^0(X,\mathcal H^i(A(n))) .
\end{align}

We have the following characterization for $A=\Z/m$.

\begin{lemma} \label{lem:H_M,nrZ/m}
Let $X$ be a smooth variety over a  field $k$ and let $m$ be an integer invertible in $k$.
Then 
$$
H^i_{M,nr}(X,\Z/m(n))=\begin{cases}
    H^i_{nr}(X_\et,\mu_{m}^{\otimes n})\quad &\text{if $i\leq n$};\\
    0\quad &\text{otherwise} .
\end{cases}
$$
\end{lemma}
\begin{proof}
If $i\leq n$, then the claim is a direct consequence of the Beilinson--Lichtenbaum conjecture, proven by Rost and Voevodsky \cite{Voevodsky}, see Theorem \ref{thm:prelim-comparison-H_m-to-H_et}, and the fact that $\Z/m(n)_{\et}\cong \mu_m^{\otimes n}$ by Geisser--Levine, see \eqref{eq:geisser-levine}.
If $i>n$, then $H^i_M(k(X),A(n))=0$
 (see \eqref{eq:Hi_M(X,n)=0-for-i})
and the result is clear.
\end{proof}

For any abelian group $A$ we also have the following description which holds without any smoothness assumption.

\begin{lemma} \label{lem:H_M-nr-lift}
Let $X$ be an equi-dimensional quasi-projective variety over a field and let $\alpha\in H^i_{M,nr}(X,A(n))$.
Then, for all $x\in X^{(1)}$, there is a representative $\alpha'\in H^i_{BM,M}(U,A(n))$ of $\alpha$ via \eqref{eq:H_BM,M-fields} for some open subset $U\subset X$ with $x\in U$.
\end{lemma}
\begin{proof}
This follows from \eqref{eq:H_BM,M-fields} and the localization sequence \eqref{eq:localization-H_BM,M} applied to a sufficiently small open subset of $X$ which contains $x$.
\end{proof}

As in \cite[Definition 5.1]{Sch-refined}, we define the refined unramified motivic cohomology as follows,  cf.~\cite[\S 4.2]{kok-zhou} and \cite[Section 2]{AS24}.

\begin{definition} \label{def:H_nr-motivic-refined} 
Let $X$ be a smooth equi-dimensional quasi-projective scheme over a field $k$ and let $A$ be an abelian group.
For $j\geq 0$, we define $H^i_M(F_jX,A(n))$ as the direct limit of $H^i_{M}(U,A(n))$ where $U$ runs through all open subsets of $X$ such that $U$ contains all codimension $j$ points of $X$.
We further define the $j$-th refined unramified motivic cohomology of $X$ as
$$
H^i_{M,j,nr}(X,A(n))\coloneq \im(H^i_M(F_{j+1}X,A(n))\to H^i_M(F_jX,A(n))) .
$$ 
\end{definition}

If $X$ is integral, there is a canonical isomorphism $H^i_M(F_0X,\Z(n))\stackrel{\cong}\to H^i_M(k(X),\Z(n))$.
In view of this, the unramified motivic cohomology identifies 
to the 0-th refined unramified motivic cohomology as follows.

\begin{proposition} \label{prop:H_j,nr}
    Let $X$ be a smooth projective variety over a field $k$.
    Then there is a canonical isomorphism $$
    H_{M,0,nr}^i(X,\Z(n))\stackrel{\cong}\longrightarrow H_{M,nr}^i(X,\Z(n)) .
    $$
\end{proposition}
\begin{proof}
The proof of \cite[Lemma 5.8]{Sch-refined} only requires a localization sequence which is compatible with respect to Zariski localization; this exists for higher Chow groups thanks to \cite{bloch-JAG}, see \eqref{eq:localization-H_BM,M}.
In particular, we get a long exact sequence
\begin{align}\label{eq:localization-H_BM,M-Fj}
\dots \longrightarrow H^i_{BM,M}(F_{j+1}X,A(n))\longrightarrow H^i_{BM,M}(F_jX,A(n))\stackrel{\del} \longrightarrow \bigoplus_{x\in X^{(j+1)}} H^{i-1-2j}_{M}(\kappa(x),A(n-j-1))\longrightarrow \dots.
\end{align}
The statement in the proposition is a direct consequence of this applied to $j=0$, because the canonical map $H^i_{BM,M}(F_0X,A(n))\to H^i_M(k(X),A(n))$ is an isomorphism, see \eqref{eq:H_BM,M-fields}.

For convenience of the reader, we give a second argument which involves the Mayer--Vietoris sequence (which in turn is a consequence of the localization sequence, respectively the corresponding exact triangle in the derived category).
We first note that there is a canonical morphism
$$
H_{M,0,nr}^i(X,\Z(n))\longrightarrow H_{M,nr}^i(X,\Z(n))
$$
which is injective because $H_M^i(F_0X,A(n))\cong H_M^i(k(X),A(n))$.
It thus suffices to prove that any unramified class 
$[\alpha]\in H^i_{M,nr}(k(X),A(n))$ lifts to a big open subset of $X$, i.e.\ to a Zariski open subset that contains all codimension one points of $X$.
To prove this, pick a representative $\alpha\in H^i_M(U,A(n))$ of $[\alpha]$ for some dense open subset $U\subset X$.
Let $S\subset  X^{(1)}$ be the set of codimension one points of $X$ that are not contained in $U$ and note that $S$ is a finite set.
If $S$ is empty, then $\alpha$ is a lift of $[\alpha]$ to $F_1X$ and we are done.
Otherwise, let $x\in S$.
By Lemma \ref{lem:H_M-nr-lift}, there is a class $\beta \in H^i_{M}(V,A(n))$ with $x\in V$ such that $\beta$ and $\alpha$ agree on some dense open subset $W\subset U\cap V$.
Up to removing from $V$ a closed subset that does not contain $x$ and removing from $U$ a closed subset of codimension $\geq 2$, we can assume $W=U\cap V$ and so $\alpha|_{U\cap V}=\beta|_{U\cap V}$.
We then consider the Mayer--Vietoris exact sequence (see e.g.\ \cite[(1.1), item (3)]{levine})
$$
\dots \longrightarrow H^i_M(U\cup V,A(n))\longrightarrow H^i_M(U,A(n))\oplus H^i_M(V,A(n)) \longrightarrow H^i_M(U\cap V,A(n)) \longrightarrow \dots 
$$
and conclude that there is a class $\gamma\in H^i_M(U\cup V,A(n))$ with $\gamma|_U=\alpha$.
Note that $X^{(1)}\setminus (U\cup V)^{(1)}\subset S\setminus \{x\}$.
Repeating this argument inductively thus yields a lift of $\alpha$ to some open subset of $X$ which contains all codimension one points of $X$, as we want.
\end{proof}

\subsection{Proof of Theorems \ref{thm:unramified-Milnor-divisible} and \ref{thm:H_M-unramified-divisible}}

\begin{proof}[Proof of Theorem \ref{thm:H_M-unramified-divisible}]
Using \eqref{eq:limit-HiM}, we reduce to the case where $k$ is the separable closure of a finitely generated subfield $k_0\subset k$.
Since $H^i_M(k(X),\Z(n))$ vanishes for $i>n$, see Lemma \ref{lem:curlyH^iZ_motivic(n)=0},   
we may also assume that $i\leq n$.
By the Chinese remainder theorem, it suffices to treat the case where $m=\ell^r$ is a power of a prime $\ell$ that is invertible in $k$.
By Lemma \ref{lem:H_M,nrZ/m}, it then suffices to prove that the natural map
$$
H_{M,nr}^i(X,\Z(n)) \longrightarrow H^i_{nr}(X_{\et},\mu_{\ell^r}^{\otimes n})
$$
is zero.
By Proposition \ref{prop:H_j,nr}, it suffices to prove that the natural map
    $$
    H_{M,0,nr}^i(X,\Z(n))\longrightarrow H_{0,nr}^i(X_\et,\mu_{\ell^r}^{\otimes n})
    $$
    is zero.
    This map factors through a natural map
    $$
    H_{M,0,nr}^i(X,\Z(n))\longrightarrow H_{0,nr}^i(X,\Z_\ell(n)) .
    $$
    By the Bloch--Kato conjecture, proven by Rost and Voevodsky \cite{Voevodsky}, $H_{0,nr}^i(X,\Z_\ell(i-1))$  is torsion-free (see e.g.\ \cite[Remark 5.14]{Sch-refined}).
    Since $k$ is separably closed, it contains all $\ell$-power roots of unity, and so $H_{0,nr}^i(X,\Z_\ell(n))\cong H_{0,nr}^i(X,\Z_\ell(i-1))$ is torsion-free as well.
    It thus suffices to prove that the natural map
    $$
    H_{M,0,nr}^i(X,\Q(n))\longrightarrow H_{0,nr}^i(X,\Q_\ell(n)) 
    $$
    is zero.
    The image of this map is an ind-mixed Galois submodule of weight $0$.
    We thus conclude by noting that the weights $w$ of the ind-mixed Galois module $H_{0,nr}^i(X,\Q_\ell(n))$ satisfy
$$
i-2n\leq w\leq \max(i-2n, 2i-2n-2) ,
$$
see Proposition \ref{prop:weights-H_j,nr}, and so $w\neq 0$, because $n\geq 1$ and $i\leq n$ by the above reduction step.
\end{proof}
 
We record  the following strengthening of the above result.

\begin{theorem} \label{thm:strengthen-thm:H_M-unramified-divisible}
Let $X$ be a smooth variety over a separably closed field $k$ and let $m$ be an integer that is invertible in $k$.
Then the following holds:
\begin{enumerate}
    \item For $i\neq n$, the natural map 
$ % \label{eq:thm:strengthen-thm:H_M-unramified-divisible}
H^i_M(k(X),\Z(n))\longrightarrow H^i_M(k(X),\Z/m(n))
$ is zero. \label{item:thm:strengthen-thm:H_M-unramified-divisible:1}
\item The natural map of Zariski sheaves $\mathcal H^i_M(\Z(n))\to \mathcal H^i_M(\Z/m(n))$ on $X$ is zero for all $i\neq n$.\label{item:thm:strengthen-thm:H_M-unramified-divisible:2}
\end{enumerate} 
\end{theorem}

\begin{proof} 
Let us first prove the vanishing in item \eqref{item:thm:strengthen-thm:H_M-unramified-divisible:1}.
Since $H^i_M(k(X),\Z(n))$ vanishes for $i>n$ (see \eqref{eq:Hi_M(X,n)=0-for-i}), we can assume $i<n$.
Using \eqref{eq:limit-HiM}, we can further reduce to the case where $k$ is the separable closure of a finitely generated field $k_0$, such that $X$ admits a model over $k_0$, cf.~beginning of the proof of Theorem \ref{thm:Hi_MXQ/Z-vanishing}.
As a consequence of Theorem \ref{thm:prelim-comparison-H_m-to-H_et}, $H^i_M(k(X),\Z/m(n))\cong H^i(F_0X,\Z/m(n))$.
Hence,
 the same reduction steps as in the proof of Theorem \ref{thm:H_M-unramified-divisible} reduce us to showing that for a prime $\ell$ invertible in $k$, the natural map
$$
H^i_M(k(X),\Z(n))\longrightarrow H^i(F_0X,\Z_\ell(n))
$$
is zero.
By the Bloch--Kato conjecture, proven by Rost and Voevodsky \cite{Voevodsky}, $H^i(F_0X,\Z_\ell(i-1))$  is torsion-free; the same holds for $H^i(F_0X,\Z_\ell(n))$, because $k$ contains all $\ell$-power roots of unity.
It thus suffices to show that $H^i_M(k(X),\Q(n))\longrightarrow H^i(F_0X,\Q_\ell(n))$ is zero; since any class in $H^i_M(k(X),\Q(n))$ can be defined over $k'(X)$ for some finitely generated extension $k'/k_0$, we find that the image of this map is an ind-mixed Galois submodule of weight zero. 
The vanishing in \eqref{item:thm:strengthen-thm:H_M-unramified-divisible:1} therefore follows from the fact that for $i<n$, the  weights $w$ of $H^i(F_0X,\Q_\ell(n))$ satisfy $w\leq 2i-2n<0$, cf.~Proposition \ref{prop:weights}.

It remains to show that the vanishing in item \eqref{item:thm:strengthen-thm:H_M-unramified-divisible:1} implies that $\mathcal H^i_M(\Z(n))\to \mathcal H^i_M(\Z/m(n))$ vanishes for $i\neq n$.
Since $m$ is invertible in $k$, \cite[Lemma 4.8]{kok-zhou} allows us to pass to the perfect closure of $k$.
We may thus assume that $k$ is perfect.
Then the Gersten conjecture holds for $\mathcal H^i_M(\Z(n))$ and  $\mathcal H^i_M(\Z/m(n))$, see  \cite[Theorem 24.11]{MVW} or \cite{bloch-motivic}.
In particular, we have a commutative diagram
$$
\xymatrix{
\mathcal H^i_M(\Z(n))\ar[r]\ar[d]& \iota_{\eta, \ast}H^i_M(k(X),\Z(n)) \ar[d]\\
\mathcal H^i_M(\Z/m(n))\ar[r]\ar[r]& \iota_{\eta, \ast}H^i_M(k(X),\Z/m(n))
}
$$
where $\iota_\eta\colon \Spec k(X)\to X$ denotes the inclusion of the generic point, and where the horizontal maps are injections by the Gersten conjecture.
By item \eqref{item:thm:strengthen-thm:H_M-unramified-divisible:1} proven above, the vertical map on the right is zero, hence so is the vertical map on the left. 
This concludes the proof of the theorem.
\end{proof}

\begin{proof}[Proof of Theorem \ref{thm:unramified-Milnor-divisible}]
By the Gersten conjecture for motivic cohomology (see \cite[\S 10]{bloch-motivic} and \cite{bloch-JAG}), we have
$H^i_{M,nr}(X,\Z(i))\cong H^0(X,\mathcal H^i(\Z(i)))$,  
where $\mathcal H^i(\Z(i))$ denotes the Zariski sheaf associated to $U\mapsto H^i_M(U,\Z(i))$, see \eqref{eq:H_nr=H^0-curlyH}.
Moreover, $\mathcal K^M_i\stackrel{\cong}\to \mathcal H^i(\Z(i))$ by \cite{kerz} and so Theorem \ref{thm:unramified-Milnor-divisible} follows from Theorem \ref{thm:H_M-unramified-divisible}. 
\end{proof}

\begin{proof}[Proof of Corollary \ref{cor:unramified-Milnor-divisible}]
This follows directly from Theorem \ref{thm:unramified-Milnor-divisible} and the Gersten conjecture  for Milnor K-theory, proven by Kerz \cite{kerz}. 
\end{proof}

\section{Proof of Theorem \ref{thm:filtration-L-intro}} \label{sec:thm:filtration-L-intro}

\subsection{The motivic coniveau filtration} \label{sec:coniveau-filtration}

Let $X$ be a smooth variety over a field.
The coniveau filtration $N^\ast$ on motivic cohomology $H^i_M(X,\Z(n))$ is defined as follows:
$$
N^cH^i_M(X,\Z(n))\coloneq \ker(H_M^i(X,\Z(n))\to H_M^i(F_{c-1}X,\Z(n))),
$$
where $i,c,n\in \Z$  and $F_cX=\emptyset$ for $c<0$.
In other words, $\alpha\in H^i_M(X,\Z(n))$ lies in $N^c$ if and only if $\alpha$ vanishes away from a closed subset of codimension $\geq c$ in $X$.\footnote{We remark that this filtration does not coincide with the coniveau filtration on Chow groups introduced by Bloch, which filters elements in the Chow group by the codimension of closed subsets on which they are homologically trivial, cf.\ \cite{bloch-coniveau,Sch-refined}.}
 
\begin{lemma}\label{lem:N^{i-n}}
Let $X$ be a smooth variety over a field.
Then $N^{i-n}H^i_M(X,\Z(n))=H^i_M(X,\Z(n))$ and $ N^{n+1}H^i_M(X,\Z(n))=0$.
\end{lemma}
\begin{proof}
Classes in $H^i_M(X,\Z(n))=\CH^n(X,2n-i)$ are represented by codimension $n$ cycles on $X\times \mathbb A^{2n-i}$; such classes vanish if we remove suitable subsets of codimension $n-2n+i=i-n$ from $X$.
Hence, $N^{i-n}H^i_M(X,\Z(n))=H^i_M(X,\Z(n))$.
To see $ N^{n+1}H^i_M(X,\Z(n))=0$, we note that $N^cH^i_M(X,\Z(n))$ is generated by images of classes in $H^{i-2c}_{BM,M}(Z,\Z(n-c))$ with $Z\subset X$ closed of pure dimension $\dim X-c$, and $H^{\ast}_{BM,M}(Z,\Z(n))=0$ for $n<0$.
This concludes the proof of the lemma.
\end{proof}

By Lemma \ref{lem:N^{i-n}}, the motivic coniveau filtration is of the form
$$
H^i_M(X,\Z(n))=N^{i-n}H^i_M(X,\Z(n))\supset N^{i-n+1}\supset N^{i-n+2}\supset \dots \supset N^{n}\supset N^{n+1}=0.
$$
This filtration compares as follows to the filtration $L_\ast$ from the hypercohomology spectral sequence, defined by  
\begin{align} \label{eq:filtration-L-body}
 L_jH^i_M(X,\Z(n))\coloneq \im(H^i(X_\zar, \tau_{\leq j}\Z(n)) \to H^i(X_\zar, \Z(n)) ) .
\end{align}  

\begin{lemma} \label{lem:L-versus-N}
Let $X$ be a smooth variety over a field $k$. Then for all integers $i,c,n\in \Z$, 
$$
N^cH^i_M(X,\Z(n))=L_{i-c}H^i_M(X,\Z(n)).
$$
\end{lemma}

\begin{proof}
Note that $L_\ast$ is increasing. We  formally define the decreasing filtration $\tilde L^{j}\coloneq L_{-j}$.
The filtration $\tilde L^\ast$ is induced by the hypercohomology spectral sequence, see \cite[(1.4.5), (1.4.6)]{deligne-HodgeII}:
$$
\tilde E_1^{p,q}=H^{2p+q}(X_{\zar},\mathcal H_M^{-p}(\Z(n)))\Longrightarrow H_M^{p+q}(X,\Z(n)) .
$$
We use the renumbering $E_{r+1}^{p,q}\coloneq \tilde E_r^{-q,p+2q}$ to turn the above spectral sequence into one that starts at $E_2$, cf.\ \cite[(1.4.8)]{deligne-HodgeII}:
$$
E_2^{p,q}=\tilde E_1^{-q,p+2q}=H^{p}(X_{\zar},\mathcal H_M^{q}(\Z(n))) .
$$
By a result of Deligne and Paranjape, see \cite[Footnote to Remark 6.4]{BO} and \cite[Corollary 4.4]{paranjape}, 
the spectral sequence $E_r^{p,q}$ agrees from $r\geq 2$ onwards with the coniveau spectral sequence.
We have 
$$
{\rm gr}_{\tilde L}^{q}H^{i}_M(X,\Z(n))=\tilde E_{\infty}^{q,i-q}\quad \text{and}\quad
{\rm gr}_N^{q}H^{i}_M(X,\Z(n))=E_\infty^{q,i-q}.
$$
Via the reindexing $E_{r+1}^{a,b}\coloneq \tilde E_r^{-b,a+2b}$ we get $E_{\infty}^{a,b}\coloneq \tilde E_{\infty}^{-b,a+2b}$ and so
$$
{\rm gr}_N^{q}H^{i}_M(X,\Z(n))={\rm gr}_{\tilde L}^{-(i-q)}H^{i}_M(X,\Z(n)).
$$ 
Hence, $N^qH^i_M(X,\Z(n))=\tilde L^{q-i}H^i_M(X,\Z(n))=L_{i-q}H^i_M(X,\Z(n))$.
This proves the lemma.
\end{proof}

\subsection{A cycle class map} \label{subsec:cycle-class} 

\begin{lemma} \label{lem:cycle-class-H_BM}
Let $X$ be an equi-dimensional quasi-projective scheme over a field $k$.
Let $\ell$ be a prime invertible in $k$ and let $A\in \{\Z_\ell,\Q_\ell,\Q_\ell/\Z_\ell,\Z/\ell^r\}$.
Then for all $i,n\in \Z$, there is a canonical cycle class map
$$
\cl\colon H^i_{BM,M}(X,A(n))\longrightarrow H^i_{BM}(X,A(n)) ,
$$
that is induced by Geisser's isomorphism \eqref{eq:lem:cycle-class-H_BM:0} below.
In particular, the cycle class map is compatible with the localization sequence and coincides with the usual \'etale cycle class map if $X$ is smooth. 
\end{lemma}
\begin{proof}
The case $A=\Z/\ell^r$ is contained in \cite[\S 4.1]{kok-zhou} and relies on the six operations in a suitable version of Voevodsky's triangulated category of motives.
The case $A=\Q_\ell/\Z_\ell$ can be deduced from this by direct limits.
Here we use Geisser's work \cite{geisser-duality} to give a more direct proof which also works for integral coefficients.
 
Topological invariance of the pro-\'etale site \cite[Lemma 5.4.2]{BS} together with \cite[Lemma 4.8]{kok-zhou} (see also \cite[Lemma 6.8]{Sch-refined}) allow us to pass to the perfect closure of $k$.
Hence, we may assume that $k$ is perfect.
In this case, let $f\colon X\to \Spec k$ be the structure map and let $d_X=\dim X$.
We then have a canonical isomorphism
$$
\Z^{BM}(n)\otimes^{\mathbb L} \Z/\ell^r\stackrel{\cong}\longrightarrow f^!(\Z^{BM}(n-d_X)\otimes^{\mathbb L} \Z/\ell^r)[-2d_X] ,
$$
in $D(X_\et)$, 
see \cite[Corollary 4.7(a)]{geisser-duality} and note that the complex $\Z^c(n)_X$ in \emph{loc.~cit.~}agrees by definition with the cycle complex $\Z^{BM}(d_X-n)[2d_X]$, cf.~\eqref{eq:A^BM}.
On $\Spec k$, we have $\Z^{BM}(n-d_X)\otimes^{\mathbb L} \Z/\ell^r\cong \mu_{\ell^r}^{\otimes{n-d_X}}$, see \eqref{eq:geisser-levine} or \cite{bloch-motivic}.
Hence, 
\begin{align} \label{eq:lem:cycle-class-H_BM:0}
\Z^{BM}(n)\otimes^{\mathbb L} \Z/\ell^r\stackrel{\cong}\longrightarrow f^!\mu_{\ell^r}^{\otimes{n-d_X}}[-2d_X] ,
\end{align}
in $D(X_\et)$.
We thus obtain a natural morphism
\begin{align} \label{eq:lem:cycle-class-H_BM:1}
\Z^{BM}(n)\otimes^{\mathbb L} \Z_\ell=\Z_\ell^{BM}(n) \longrightarrow R\lim f^!\mu_{\ell^r}^{\otimes{n-d_X}}[-2d_X].
\end{align}
Let now $\nu\colon X_\proet\to X_\et$ be the natural map of sites.
By \cite[Lemma 6.7.19]{BS}, we have
$ 
f^!\widehat \Z_\ell(n) \cong R\lim \nu^\ast f^! \Z/\ell^r(n) .
$ 
Applying $\nu_\ast$, we then get
$$
\nu_\ast f^!\widehat \Z_\ell(n-d_X) \cong R\lim \nu_\ast\nu^\ast f^! \mu_{\ell^r}^{\otimes{n-d_X}}\cong R\lim f^! \mu_{\ell^r}^{\otimes{n-d_X}},
$$
where the first isomorphism uses that  $\nu_\ast$ commutes with $R\lim$ and the second isomorphism uses that $\nu_\ast\nu^\ast \cong \id$ on bounded complexes, see \cite[Corollary 5.1.6]{BS}.
Combining this with \eqref{eq:lem:cycle-class-H_BM:1}, we obtain the cycle class map in the case $A=\Z_\ell$ after applying $R\Gamma(X_{\et},-)$ and composition with the natural map $H^i_{BM,M}(X,A(n))\to H^i(X_\et,A^{BM}(n))$.
Compatibility with the localization sequence can be deduced from \cite[Proposition 3.5(a)]{geisser}; 
compatibility with the usual \'etale cycle class map if $X$ is smooth follows from the construction. 
The case $A=\Q_\ell$ follows after $\otimes \Q_\ell$ and the cases $A\in\{\Z/\ell^r,\Q_\ell/\Z_\ell\}$ follow via similar arguments as above.
\end{proof}

\subsection{Proof of Theorem \ref{thm:filtration-L-intro}} 
Let $X$ be an equi-dimensional quasi-projective variety over a field $k$. 
Let $\ell$ be a prime invertible in $k$.
Recall that there is a cycle class map
\begin{align} \label{eq:cl-H_BM}
\cl\colon H^i_{BM,M}(X,\Z/\ell^r(n))\longrightarrow H^i_{BM}(X,\Z/\ell^r(n)) ,
\end{align}
which is compatible with the localization sequence and coincides with the usual \'etale cycle class map if $X$ is smooth, see \cite[Proposition 4.9]{kok-zhou} or Lemma \ref{lem:cycle-class-H_BM}. 
It follows from the Beilinson--Lichtenbaum conjectures (see Theorem \ref{thm:prelim-comparison-H_m-to-H_et}) that, for $X$ smooth, the cycle class map \eqref{eq:cl-H_BM} is an isomorphism for $i\leq n$ and injective for $i=n+1$.
Compatibility with the localization sequence then allowed Kok and Zhou to prove the following via the five lemma:

\begin{proposition}[{\cite[Proposition 4.9]{kok-zhou}}] \label{prop:kok-zhou}
Let $X$ be an equi-dimensional quasi-projective variety over a field $k$.  
Let $\ell$ be a prime invertible in $k$.  
Then the cycle class map  \eqref{eq:cl-H_BM} is an isomorphism for $i\leq n$ and injective for $i=n+1$.
\end{proposition}  

Using this, we can prove the following.

\begin{theorem} \label{thm:Hi_MXQ/Z-vanishing}
Let $k$ be a separably closed field and let $\ell$ be a prime invertible in $k$.
Let $X$ be an equi-dimensional quasi-projective variety over $k$ and let $i,n$ be integers with $i<n$.
Then $H^i_{BM,M}(X,\Z(n))\otimes \Q_\ell/\Z_\ell=0$.
\end{theorem}
\begin{proof}
Since $H^i_{BM,M}(X,\Z(n))$ agrees with Bloch's higher Chow groups, each class is defined over a finitely generated field.
This allows us to reduce to the case where $X=X_0\times_{k_0}k$ is defined over a finitely generated field $k_0$ and $k$ is the separable closure of $k_0$.

By Lemma \ref{lem:cycle-class-H_BM}, we have a cycle class map
\begin{align} \label{eq:cl-H_BM-2}
\cl\colon H^i_{BM,M}(X,\Z(n))\longrightarrow H^i_{BM}(X,\Z_\ell(n)) .
\end{align}
Since $k$ is separably closed and $X$ is of finite type over $k$, %$H^i_{BM}(X,\Z/\ell^r(n))$ is finite and 
the target
$H^i_{BM}(X,\Z_\ell(n))$ is finitely generated as a $\Z_\ell$-module.
To see this, let $f\colon X\to \Spec k$ be the structure map and note that $Rf_\ast f^!\widehat \Z_\ell(n-d_X)$ is constructible \cite[\S 6.7]{BS}, hence is quasi-isomorphic to a perfect complex of $\Z_\ell$-modules (see \cite[Proposition 6.6.11]{BS}) and hence to a complex of finitely generated $\Z_\ell$-modules.

Let $G=\Gal(k/k_0)$.
By Proposition \ref{prop:weights}, $H^i_{BM}(X,\Q_\ell(n))$ is a mixed $\Q_\ell$-$G$-module of weights $w\leq 2i-2n<0$.
Hence, the cycle class map in \eqref{eq:cl-H_BM-2} is torsion.
Since $H^i_{BM}(X,\Z_\ell(n))$ is finitely generated as a $\Z_\ell$-module, there is an integer $N$ such that the image of
$$
H^i_{BM,M}(X,\Z(n))\longrightarrow H^i_{BM}(X,\Z/\ell^r(n))
$$
is $N$-torsion for all $r\geq 0$.
Taking direct limits, we find that the natural map
$$
H^i_{BM,M}(X,\Z(n))\otimes \Q_\ell/\Z_\ell \longrightarrow H^i_{BM}(X,\Q_\ell/\Z_\ell(n))
$$
is $N$-torsion as well.
Since the source of this map is divisible, we find that the map is in fact zero.
This map factors as follows
$$
H^i_{BM,M}(X,\Z(n))\otimes \Q_\ell/\Z_\ell \hookrightarrow H^i_{BM,M}(X,\Q_\ell/\Z_\ell(n)) \longrightarrow H^i_{BM}(X,\Q_\ell/\Z_\ell(n)) .
$$
The first arrow in this sequence is injective (for all values of $i,n$), because of the long exact sequences associated to the coefficient sequences $0\to \Z \to \Z\to \Z/\ell^r\to 0$ and because the direct limit functor is exact.
 The second arrow in the above sequence is injective for $i\leq n+1$, as can be seen by applying direct limits to the injection in Proposition \ref{prop:kok-zhou}, cf.~\eqref{eq:H_BM,M-Ql/Zl}.
Hence, the composition is injective.
However, as we have seen above, the composition is also zero, which implies $H^i_{BM,M}(X,\Z(n))\otimes \Q_\ell/\Z_\ell=0 $, as we want.
\end{proof}

We have the following analogue over finite fields.

\begin{theorem} \label{thm:Hi_MXQ/Z-vanishing-finite}
Let $k$ be a finite field and let $\ell$ be a prime invertible in $k$.
Let $X$ be an equi-dimensional quasi-projective variety over $k$ and let $i,n$ be integers with $i<n$.
Then $H^i_{BM,M}(X,\Z(n))\otimes \Q_\ell/\Z_\ell=0$.
\end{theorem}
\begin{proof}
As in the proof of Theorem \ref{thm:Hi_MXQ/Z-vanishing}, we have natural maps 
$$
H^i_{BM,M}(X,\Z(n))\otimes \Q_\ell/\Z_\ell \hookrightarrow H^i_{BM,M}(X,\Q_\ell/\Z_\ell(n)) \longrightarrow H^i_{BM}(X,\Q_\ell/\Z_\ell(n)) ,
$$
where the first arrow is always injective and the second arrow, induced by \eqref{eq:H_BM,M-Ql/Zl} and \eqref{eq:cl-H_BM}, is injective for $i\leq n+1$, see Proposition \ref{prop:kok-zhou}.
It thus suffices to show that the above composition is trivial.
This map factors through $H^i_{BM}(X,\Z_\ell(n))\otimes_{\Z_\ell} \Q_\ell/\Z_\ell $ and so it suffices to show that the latter group is trivial.
This group is divisible and so it suffices to show that $H^i_{BM}(X,\Z_\ell(n))$ is finite for $i<n$, which is proven in Proposition \ref{prop:H_BM-vanishing-finite}.
This concludes the proof.
\end{proof}

Let $X$ be a quasi-projective variety.
The coniveau filtration $N^\ast$ on $H^i_{BM,M}(X,\Z(n))$ is defined as follows: $\alpha\in H^i_{BM,M}(X,\Z(n))$ lies in $N^c$ if and only if $\alpha$ vanishes away from a closed subset of codimension $\geq c$ in $X$.
If $X$ is smooth, then this agrees with the coniveau filtration on $H^i_{M}(X,\Z(n))$. 
We are finally in the position to prove Theorems \ref{thm:filtration-L-intro} and \ref{thm:filtration-intro-BM}, stated in the introduction.

\begin{proof}[Proof of Theorem \ref{thm:filtration-L-intro}]
By Lemma \ref{lem:L-versus-N}, assertion \eqref{eq:L^jH^iotimesQl/Zl} is equivalent to \eqref{eq:L^jH^iotimesQl/Zl-N}; it thus suffices to prove the latter.
There is a surjection
$$
\lim_{\substack{\longrightarrow\\ Z\subset X}} H_{BM,M}^{i-2j}(Z,\Z(n-j))\longrightarrow N^jH_{BM,M}^i(X,\Z(n)) , 
%\subset H_{BM,M}^i(X,\Z(n))
$$
where $Z\subset X$ runs through all closed equi-dimensional subschemes of pure codimension $j$.
By the right exactness of the tensor product, it thus suffices to prove that
$$
\left(\lim_{\substack{\longrightarrow\\ Z\subset X}} H_{BM,M}^{i-2j}(Z,\Z(n-j))\right)\otimes \Q_\ell/\Z_\ell \cong \lim_{\substack{\longrightarrow\\ Z\subset X}} H_{BM,M}^{i-2j}(Z,\Z(n-j))\otimes \Q_\ell/\Z_\ell 
$$
vanishes.
This follows from Theorems \ref{thm:Hi_MXQ/Z-vanishing} and \ref{thm:Hi_MXQ/Z-vanishing-finite}, because $j>i-n$ is equivalent to $i-2j<n-j$.
\end{proof}

\begin{proof}[Proof of Theorem \ref{thm:filtration-intro-BM}]
This follows, by the same argument as above, from Theorems \ref{thm:Hi_MXQ/Z-vanishing} and \ref{thm:Hi_MXQ/Z-vanishing-finite}.
\end{proof}

\begin{proof}[Proof of Corollary \ref{cor:filtration-L-intro-1}]
This is an immediate consequence of Theorem \ref{thm:filtration-L-intro} and the main result in \cite{kerz}.
\end{proof}

By Corollary \ref{cor:filtration-L-intro-1}, any smooth quasi-projective equi-dimensional scheme $X$ over a separably closed field satisfies 
\begin{align} \label{eq:before-example-sharp}
H^i_M(X,\Z(n))\otimes \Q_\ell/\Z_\ell=0 \quad \text{for $i<n$.}
\end{align}
The same result holds for finite fields by item \eqref{item:thm:Hi_MXQ/Z-vanishing-finite:2} in Theorem \ref{thm:filtration-intro-BM}.
The following example shows that these results are sharp. 

\begin{example} \label{example:sharp-1}
Let $k$ be a field and $\mathbb G_m=\mathbb A^1_k\setminus \{0\}$.
Then 
\begin{align} \label{eq:HiM(Gm^d)}
H^i_M((\mathbb G_m)^d,\Z(i))\otimes \Q_\ell/\Z_\ell \neq 0\quad \quad \text{for all $0\leq i\leq d $.}
\end{align}
The case $i=0$ is trivial and we may assume $i\geq 1$.
To see this, let $X$ be a smooth variety over $k$ and consider the long exact localization sequence 
$$
H^i_M(X\times \mathbb G_m,\Z(i)) \stackrel{\partial} \longrightarrow H^{i-1}_M(X ,\Z(i-1))\stackrel{\iota_\ast} \longrightarrow H^{i+1}_M(X\times \mathbb A_k^1 ,\Z(i)) .
$$
The composition of $\iota_\ast$ with the restriction to $X\times \{1\}$ is zero.
Since the restriction map $H^{i+1}_M(X\times \mathbb A_k^1 ,\Z(i)) \to H^{i+1}_M(X\times \{1\} ,\Z(i)) $ is an isomorphism by $\mathbb A^1$-homotopy invariance (see e.g.~\cite[p.~269, (ii)]{bloch-motivic}), we conclude that $\iota_\ast$ is zero and $\partial$ is surjective.
Combining this with the fact that $H^1_M(X,\Z(1))=H^0(X,\mathbb G_m)$, we conclude \eqref{eq:HiM(Gm^d)} by induction on $i$.
\end{example}

The next example shows that the vanishing in Theorem \ref{thm:filtration-L-intro} and in item \eqref{item:thm:Hi_MXQ/Z-vanishing-finite:1} of Theorem \ref{thm:filtration-intro-BM} are sharp as well.

\begin{example} \label{example:sharp-2}
There is a smooth quasi-projective variety $X$ such that for all $i,n\geq 1$ with $n\leq i\leq 2n$ we have
\begin{enumerate}
    \item $H^i_M(X,\Z(n))\otimes \Q_\ell/\Z_\ell \neq 0 $;
    \item $N^c H^i_M(X,\Z(n))\otimes \Q_\ell/\Z_\ell \neq 0 $ for $c=i-n$.
\end{enumerate} 
To see this, define $c\coloneq i-n$.
Then $n-c\geq 0$ and Example \ref{example:sharp-1} yields the existence of a smooth quasi-projective variety $Z$ with $H^{n-c}_M(Z,\Z(n-c))\otimes \Q_\ell/\Z_\ell \neq 0$.
Then, $X=\mathbb P^c\times Z$ satisfies the above non-vanishing properties by the projective bundle formula, see e.g.~\cite[p.~269, (iv)]{bloch-motivic}.
\end{example}

Finally, let us mention the following, which shows that, in contrast to Example \ref{example:sharp-1}, the vanishing result in \eqref{eq:before-example-sharp} can be improved under suitable assumptions on a smooth compactification, see also the results in Appendix \ref{sec:appendix-B} below.

\begin{proposition}
Let $X$ be a smooth quasi-projective equi-dimensional scheme over a field $k$, which is either separably closed or finite.
Assume that $X$ admits a smooth projective compactification $X\subset X^c$, such that $X^c\setminus X$ has codimension at least two in $X^c$.
Then
$$
H^i_M(X,\Z(n))\otimes \Q_\ell/\Z_\ell=0 \quad \text{for $i\leq n$.}
$$
\end{proposition}
\begin{proof}
Since $X$ is smooth, $H^i_M(X,\Z(n))=H^i_{BM,M}(X,\Z(n))$.

Let us first assume that $k$ is separably closed.
As in the proof of Theorem \ref{thm:Hi_MXQ/Z-vanishing}, we reduce to the case where $k$ is the separable closure of a finitely generated field $k_0\subset k$.
Let $i\leq n$.
Following the proof of Theorem \ref{thm:Hi_MXQ/Z-vanishing}, it then suffices to show that $H^i(X,\Q_\ell(n))$ has negative weights, which follows from Proposition \ref{prop:weights} together with the localization sequence applied to the inclusion $X\subset X^c$.

The case where $k$ is finite follows by similar arguments as in the proof of Theorem \ref{thm:Hi_MXQ/Z-vanishing-finite} from the fact that $H^i(X,\Z_\ell(n))$ is finite for $i\leq n$. 
The latter follows in turn from the same argument as in the proof of Proposition \ref{prop:H_BM-vanishing-finite}, together with the fact that $H^i(X_{\bar k},\Q_\ell(n))$ is a mixed Galois-module of negative weights for $i\leq n$, because $X^c\setminus X$ has codimension at least two in $X^c$.
This concludes the proof.
\end{proof}

\section{Divisibility phenomena of some Bloch--Ogus groups} \label{sec:factoring-the-map}

The main result in this section is the following 

\begin{theorem}\label{thm:Hi-curlyHj}
Let $X$ be a smooth quasi-projective equi-dimensional scheme over a separably closed field $k$.
Then the image of the natural map
$$
H^i(X_\zar,\mathcal H^j_M(\Z(n)))\longrightarrow \lim_{\substack{\longleftarrow\\ r}} H^i(X_\zar, \mathcal H^j_M(\Z/\ell^r(n)))
$$
\begin{enumerate}
\item is zero for $j\neq n$; \label{item:thm:Hi-curlyHj:1}  
\item is torsion for $j=n \geq 2$ and $i\in \{0,1\}$ if $X$ is smooth projective.
\label{item:thm:Hi-curlyHj:2}  
\end{enumerate}  
\end{theorem}

\subsection{Around the cycle class map in $\ell$-adic pro-\'etale cohomology} 
Let $X$ be a smooth quasi-projective scheme over a field $k$ and $\pi^\et \colon X_\et\to X_\zar$ and $\pi^\proet \colon X_\proet\to X_\zar$ be the natural maps of sites.

\begin{lemma} \label{lem:proet-pushforward}
Let $\ell$ be a prime invertible in $k$.
We have the following canonical identifications in $D(X_\zar)$:
$$
 \R \lim \R \pi^\et_\ast \Z/\ell^r(n)_{\et}\cong
\R \pi^\et_\ast  \R \lim \Z/\ell^r(n)_{\et}\cong \R \pi_\ast^{\proet}\widehat \Z_\ell(n) .
$$
\end{lemma}

\begin{proof}
The first isomorphism follows from the fact that $\R \lim$ and $\R \pi^\et_\ast$ commute, which in turn follows from Grothendieck's composed functor spectral sequence and the fact that $\lim\colon \Ab^\N\to \Ab$ and $\pi^\et_\ast\colon \Ab(X_\et)\to \Ab(X_\zar)$ take injectives to injectives, see e.g.\ \cite[Proof of Lemma A.1]{Sch-moving}.

It remains to prove the second isomorphism.
Let $\nu\colon X_\proet\to X_\et$ be the natural map of sites.
Since the adjunction map $\id\to \R \nu_\ast \nu^\ast$ is an equivalence (see \cite[Proposition 5.2.6.(2)]{BS}), we have
\begin{align*} % \label{eq:Rpi-Rlim}
\R \nu_\ast \nu^\ast \Z/\ell^r(n)_\et\cong   \Z/\ell^r(n)_\et .
\end{align*}
We then get
\begin{align} \label{eq:Rpi-Rlim}
\R \pi^\et_\ast  \R \lim \Z/\ell^r(n)_{\et}\cong \R \pi^\et_\ast  \R \lim \R  \nu_\ast \nu^\ast \Z/\ell^r(n)_\et\cong  \R \pi^\et_\ast \R \nu_\ast  \R \lim   \nu^\ast \Z/\ell^r(n)_\et ,
\end{align}
where we used that $\R \lim$ and $\R\nu_\ast $ commute by the Grothendieck spectral sequence and an argument as before, cf.\ \cite[Proof of Lemma A.1]{Sch-moving}. 
Recall further that $\Z/\ell^r(n)_\et \cong \mu_{\ell^r}^{\otimes n}$ by \cite{geisser-levine}, see \eqref{eq:geisser-levine}.  
Hence,  
$$
\widehat \Z_\ell(n)=\lim   \nu^\ast  \mu_{\ell^r}^{\otimes n}\cong  \R \lim   \nu^\ast  \mu_{\ell^r}^{\otimes n}\cong  \R \lim   \nu^\ast \Z/\ell^r(n)_\et ,
$$
where the first equality holds by definition and the second one follows from \cite[Propositions 3.1.10, 3.2.3, and 4.2.8]{BS}. 
If we plug this into \eqref{eq:Rpi-Rlim} and use $\pi^\proet=\pi^\et\circ \nu$, we get 
$$
\R \pi^\et_\ast  \R \lim \Z/\ell^r(n)_{\et}\cong \R \pi^\et_\ast \R \nu_\ast \widehat \Z_\ell(n)\cong \R \pi_\ast^{\proet} \widehat \Z_\ell(n).
$$
This proves the lemma.
\end{proof}

We consider the following composition of natural maps in $D(X_\zar)$:
\begin{align}  \label{eq:H_m-to-H_proet}
\Z(n)_{\zar}\to \R \lim \Z/\ell^r(n)_{\zar} \to \R \lim \R \pi^\et_\ast \Z/\ell^r(n)_{\et} \cong  \R \pi_\ast^{\proet} \widehat \Z_\ell(n)  .
\end{align}
Here, the first map is induced by the reduction modulo $\ell^r$ map $\Z(n)_{\zar}\to  \Z/\ell^r(n)_{\zar}$, the second map is induced by the natural adjunction map $\Z/\ell^r(n)_{\zar} \to \R \pi^\et_\ast (\pi^\et)^\ast\Z/\ell^r(n)_{\zar}=\R \pi^\et_\ast \Z/\ell^r(n)_{\et} $ and the identification $\R \lim \R \pi^\et_\ast \Z/\ell^r(n)_{\et} \cong  \R \pi_\ast^{\proet} \widehat \Z_\ell(n)$ is taken from Lemma \ref{lem:proet-pushforward} above. 
For each open subset $U\subset X$, there is a commutative diagram 
\begin{align} \label{eq:diagramm-HiM}
\xymatrix{
H^j_M(U,\Z(n))\ar[r]\ar[d]& H^j(U_{\proet}, \widehat \Z_\ell(n) ) \ar[d]\\
H^j_M(U,\Z/\ell^r(n))\ar[r]&
H^j(U_{\et},\mu_{\ell^r}^{\otimes n} )
}
\end{align}
where the horizontal map is induced by \eqref{eq:H_m-to-H_proet}, the vertical maps are the reduction maps and the lower horizontal map is the \'etale cycle class map.
The above diagram induces an analogous diagram of sheaves of abelian groups in the Zariski site of $X$. 
Taking cohomology and inverse limits over $r$, this induces the following commutative diagram 
\begin{align} \label{eq:diagramm-Hi-curlyHj}
\xymatrix{
H^i(X_\zar,\mathcal H^j_M(\Z(n)))\ar[r]\ar[d]& H^i(X_\zar,\R^j \pi^\proet_\ast \widehat \Z_\ell(n) )\ar[d]\\
\lim_r
H^i(X_\zar, \mathcal H^j_M(\Z/\ell^r(n))) \ar[r]&
\lim_r
H^i(X_\zar, \R^j \pi^\et_\ast \Z/\ell^r(n)_\et) 
}
\end{align}
where $\lim_r$ denotes the inverse limit over $r$. 
%We then have the following vanishing result.

\subsection{Proof of Theorem \ref{thm:Hi-curlyHj}} 
We are now in the position to prove Theorem \ref{thm:Hi-curlyHj}, stated above.

\begin{proof}[Proof of Theorem \ref{thm:Hi-curlyHj}]
Item \eqref{item:thm:Hi-curlyHj:1} follows from Theorem \ref{thm:strengthen-thm:H_M-unramified-divisible}, which asserts that the reduction map $\mathcal H^j_M(\Z(n))\to \mathcal H^j_M(\Z/m(n))$ is zero for $j\neq n$.

It remains to deal with the case where $X$ is smooth projective, $j=n$ and $i\in \{0,1\}$. 
 
Note that $k$ is the direct limit of separable closures of finitely generated fields.
A limit argument (based on the Gersten resolution of $\mathcal H^n_M(\Z(n))$, see \cite{bloch-motivic}) then reduces us to the case where $k$ is the separable closure of a finitely generated field $k_0$. 
(Here we do allow non-perfect fields $k$, over which the results from \cite{bloch-motivic} still apply; this could be avoided if we were willing to work with algebraic closures of finitely generated fields, instead of separable closures.) 

Since $j= n$, the lower horizontal map in \eqref{eq:diagramm-Hi-curlyHj} is an isomorphism (see Theorem \ref{thm:prelim-comparison-H_m-to-H_et}).
Moreover, by the Gersten conjecture for motivic cohomology (see \cite{bloch-motivic}), any class in $H^i(X_\zar,\mathcal H^n_M(\Z(n)))$ is defined over a finitely generated field and hence maps to an element of weight zero in 
$$
H^i(X_\zar, \R^n \pi^{\proet}_\ast \widehat \Q_\ell(n))=H^i(X_\zar, \R^n \pi^{\proet}_\ast \widehat \Z_\ell(n))\otimes_{\Z_\ell}\Q_\ell .
$$   
 
Proposition \ref{prop:weights-H^p(X,curly-H^q)} shows that 
$ 
H^i(X_\zar, \R^n \pi^{\proet}_\ast \widehat \Q_\ell(n))
$ 
is ind-mixed and its weights $w$ satisfy
$$
i-n\leq w\leq \max(i-n, -2) ,
$$ 
because $i\in \{0,1\}$.
Hence, $w<0$, because $i<n$.
This concludes the proof of the theorem.
\end{proof}

\subsection{Alternative proof of Theorem \ref{thm:H_M-unramified-divisible}} \label{subsec:thm:unramified-Milnor-divisible-2}
As an application of Theorem \ref{thm:Hi-curlyHj}, we give an alternative proof of Theorem \ref{thm:H_M-unramified-divisible}. 
We begin with some consequences of the work of Rost and Voevodsky \cite{Voevodsky}.

\begin{proposition}\label{prop:Tate-module}
Let $K$ be a field and let $\ell$ be a prime invertible in $K$. 
Then, for any $j\geq 0$, the group 
$$
\lim_{\substack{\longleftarrow\\ r}} H^j_{\et}(K,\mu_{\ell^r}^{\otimes j-1}) 
$$
is torsion-free.
\end{proposition}
\begin{proof}
By the Bloch--Kato conjecture, proven by Rost and Voevodsky \cite{Voevodsky}, together with the Bockstein-sequence, we see that the natural map $H^j_{\et}(K,\mu_{\ell^r}^{\otimes j-1})\to H^j_{\et}(K,\mu_{\ell^{r+1}}^{\otimes j-1})$ that is induced by the inclusion $ \mu_{\ell^r}^{\otimes j-1}\hookrightarrow \mu_{\ell^{r+1}}^{\otimes j-1}$ is injective.
This induces a canonical isomorphism 
$$
H^j_{\et}(K,\mu_{\ell^r}^{\otimes j-1})\stackrel{\cong}\longrightarrow H^j_{\et}(K,\Q_\ell/\Z_\ell(j-1))[\ell^r] .
$$
Hence, 
$$
\lim_{\substack{\longleftarrow\\ r}} H^j_{\et}(K,\mu_{\ell^r}^{\otimes j-1}) 
$$
identifies to the Tate module of $H^j_{\et}(K,\Q_\ell/\Z_\ell(j-1))$, which, as any Tate module, is torsion-free, as we want.
\end{proof}

\begin{proposition} \label{prop:torsion-free}
Let $K$ be a field and let $\ell$ be a prime invertible in $K$.
Assume that $K$ contains all $\ell$-power roots of unity.
Then for any $n,j\geq 0$, the group 
$$
\lim_{\substack{\longleftarrow\\ r}} H^j_M(K,\Z/\ell^r(n))  
$$
is torsion-free.
\end{proposition}
\begin{proof}
For $j>n$, the group $H^j_M(K,\Z/\ell^r(n))$ vanishes, see Lemma \ref{lem:curlyH^iZ_motivic(n)=0}. 
For $j\leq n$, the natural map
$$
H^j_M(K,\Z/\ell^r(n))\stackrel{\cong}\longrightarrow H^j_{\et}(K,\mu_{\ell^r}^{\otimes n})
$$
is an isomorphism by the Beilinson--Lichtenbaum conjecture, proven by Rost and Voevodsky, see Theorem \ref{thm:prelim-comparison-H_m-to-H_et} and the result of Geisser--Levine in \eqref{eq:geisser-levine}.
Since $K$ contains all $\ell^r$-th roots of unity, $\mu_{\ell^r}^{\otimes n}\cong \mu_{\ell^r}^{\otimes j-1}$.
The result therefore follows from Proposition \ref{prop:Tate-module}.
\end{proof}

We are now in the position to prove the following.

\begin{theorem} \label{thm:cor:Hi-curlyHj-1}
Let $k$ be a separably closed field and let $\ell$ be a prime invertible in $k$.
Let $X$ be a smooth projective variety over $k$.
Then for any $n\geq 1$ and $j,r\geq 0$, the map 
$$
H^0(X_\zar,\mathcal H^j_M(\Z(n)))\longrightarrow  H^0(X_\zar,\mathcal H^j_M(\Z/\ell^r(n)))
$$
that is induced by the reduction modulo $\ell^r$ map $\Z(n)\to \Z/\ell^r(n)$ vanishes.
\end{theorem}
\begin{proof}
By Theorem \ref{thm:Hi-curlyHj}, it suffices to show that 
$$
\lim_{\substack{\longleftarrow\\ r}} H^0(X_\zar,\mathcal H^j_M(\Z/\ell^r(n)))
$$
is torsion-free.
By the Gersten conjecture for motivic cohomology with $\Z/\ell^r$-coefficients (see e.g.\ \cite[Theorem 24.11]{MVW}), the natural map
$$
H^0(X_\zar,\mathcal H^j_M(\Z/\ell^r(n)))\longrightarrow H^j_M(k(X),\Z/\ell^r(n))
$$
is injective.
Since the inverse limit functor is left exact, this induces an injection
$$
\lim_{\substack{\longleftarrow\\ r}} H^0(X_\zar,\mathcal H^j_M(\Z/\ell^r(n)))\hookrightarrow \lim_{\substack{\longleftarrow\\ r}} H^j_M(k(X),\Z/\ell^r(n)) .
$$
The result thus follows from Proposition \ref{prop:torsion-free}.
\end{proof}

By the Gersten conjecture for motivic cohomology, Theorem \ref{thm:cor:Hi-curlyHj-1} implies the following,  which by the Chinese remainder theorem and Lemma \ref{lem:H_M,nrZ/m} implies Theorem \ref{thm:H_M-unramified-divisible}.

\begin{corollary}
In the notation of Theorem \ref{thm:cor:Hi-curlyHj-1}, the reduction modulo $\ell^r$ map
$ 
H^i_{M,nr}(X,\Z(n))\to H^i_{M,nr}(X,\Z/\ell^r(n))
$ 
is zero.
\end{corollary}
 
\appendix

\section{Chow groups tensor $\Q_{\ell}/\Z_{\ell}$ may be large} 
\label{sec:Chow-tensor-Q/Z}
The following result follows from \cite{totaro-chow}, which in turn relies on a theorem of Bloch--Esnault \cite{BE} and earlier results of Schoen \cite{schoen-modn} and Rosenschon--Srinivas \cite{RS}.

\begin{theorem} \label{thm:Chow-tensor-Ql/Zl}
Let $X\coloneq JC$ be the Jacobian of a very general complex projective curve $C$ of genus $3$.
Then, for any subgroup $M\subset \CH^2(X)$ with finitely generated cokernel, we have $M\otimes \Q_\ell/\Z_\ell \neq 0$  for any prime $\ell$.
\end{theorem}

\begin{proof} 
Totaro showed in \cite{totaro-chow} that $\CH^2(X)/\ell$ is infinite for all primes $\ell$.
In the process of the proof, Totaro showed that there is an integer $m$ such that for all $r$ we have
$$
\ell^m\cdot N^1H^3(X_\et,\Z/\ell^r(2))\subset \Theta \cdot H^1(X_\et,\Z/\ell^r(1)),
$$
see \cite[page 368]{totaro-chow}.
Here, $\Theta \cdot H^1(X_\et,\Z/\ell^r(1))$ denotes the image of the map $H^1(X_\et,\Z/\ell^r(1))\to H^3(X_\et,\Z/\ell^r(2)) $ given by multiplication with the first Chern class of the theta divisor of $X$ and $N^\ast$ denotes the coniveau filtration.
Since the reduction modulo $\ell^r$ map $H^1(X_\et,\Z_\ell(1))\to H^1(X_\et,\Z/\ell^r(1)) $ is surjective, this image is contained in
$$
N^1H^3(X_\et,\Z_\ell(2))\otimes \Z/\ell^r = \im(N^1H^3(X_\et,\Z_\ell(2))\to H^3(X_\et,\Z/\ell^r(2)) ) .
$$
Hence,
there is an integer $m$ such that for all $r$ we have
$$
\ell^m\cdot N^1H^3(X_\et,\Z/\ell^r(2))\subset N^1H^3(X_\et,\Z_\ell(2))\otimes \Z/\ell^r .
$$
This implies that $\ell^m$ kills the cokernel of the natural map
$$
N^1H^3(X_\et,\Q_\ell(2)) \longrightarrow N^1H^3(X_\et,\Q_\ell/\Z_\ell(2)) .
$$
It thus follows from \cite[\S 18]{MS} (see also \cite[Proposition 7.16 and Theorem 7.19]{Sch-refined}) that $\ell^m$ kills the $\ell$-primary torsion subgroup $\Griff^2(X)[\ell^\infty]$ of the Griffiths group $\Griff^2(X)$ of homologically trivial codimension 2 cycles modulo algebraic equivalence.
In other words,
$$
\Griff^2(X)[\ell^\infty]=\Griff^2(X)[\ell^m].
$$
By \cite[\S 18]{MS}, the latter is isomorphic to $N^1H^3(X_\et,\Z/\ell^m(2))/N^1H^3(X_\et,\Z_\ell(2))\otimes \Z/\ell^m$ (see also \cite[Proposition 7.16 and Theorem 7.19]{Sch-refined}), which is a finite group.

Let $A^2(X)\coloneq \CH^2(X)/\sim_{\alg}$ be the Chow group of codimension $2$ cycles modulo algebraic equivalence.
Since the subgroup of algebraically trivial cycles over the field $k=\C$ of complex numbers is divisible, $A^2(X)/\ell\cong \CH^2(X)/\ell$ is infinite by \cite{totaro-chow}.
By definition, $\Griff^2(X)\hookrightarrow A^2(X)$ is a finite index subgroup and $A^2(X)[\ell^\infty]\subset \Griff^2(X)[\ell^\infty]$ because the cohomology of an abelian variety is torsion-free.
The previous paragraph thus shows that there is an infinite dimensional subspace $V\subset A^2(X)/\ell$, such that no nontrivial element of $V$ is the reduction of an $\ell$-primary torsion element in $A^2(X)$. 

To prove the theorem, let us now assume for a contradiction that there is a subgroup $M\subset \CH^2(X)$ whose cokernel $Q$ is finitely generated and such that $M\otimes_\Z \Q_\ell/\Z_\ell=0$.
Mapping this to the Chow group  modulo algebraic equivalence $A^2(X)=\CH^2(X)/\sim_{\alg}$, we get a short exact sequence
$$
0\longrightarrow \bar M\longrightarrow A^2(X)\longrightarrow \bar Q\longrightarrow 0
$$
where $\bar Q$ is a finitely generated abelian group and where $M\twoheadrightarrow \bar M$ is surjective.
Since $\otimes_\Z \Q_\ell/\Z_\ell$ is right exact, $\bar  M\otimes_\Z \Q_\ell/\Z_\ell=0$ and $A^2(X)\otimes_\Z \Q_\ell/\Z_\ell\cong \bar  Q\otimes_\Z \Q_\ell/\Z_\ell$.

Since the above vector space $V$ is infinite-dimensional and $\bar  Q$ is finitely generated, there is a nonzero class $\alpha\in V\subset A^2(X)/\ell$ such that $\alpha$ maps to zero in $\bar  Q\otimes_\Z \Q_\ell/\Z_\ell$.
Since $A^2(X)\otimes_\Z \Q_\ell/\Z_\ell\cong \bar  Q\otimes_\Z \Q_\ell/\Z_\ell$, we find that $\alpha$ maps to zero in 
$$
A^2(X)\otimes_\Z \Q_\ell/\Z_\ell\cong \lim_{\substack{\longrightarrow \\ r}}A^2(X)/\ell^r .
$$ 
This implies that there is some integer $s$ such that $\ell^s\alpha=0$ in $A^2(X)/\ell^{s+1}$.
Let $\alpha'\in A^2(X)$ be a lift of $\alpha$, then we find that there is a class $\beta\in A^2(X)$ with
$\ell^s \alpha'=\ell^{s+1}\beta$.
Hence, $\alpha'-\ell \beta$ is $\ell^s$-torsion in $A^2(X)$.
But the class $\alpha'-\ell \beta$ agrees with $\alpha$ modulo $\ell$ and we get that $\alpha=0$, because $V$ does not contain a nonzero element which is the reduction of a torsion class.
This contradicts the fact that $\alpha$ is nonzero and hence finishes the proof. 
\end{proof} 

\begin{corollary} \label{cor:Chow-tensor-Ql/Zl}
Let $n\geq 3$.
Then there is a smooth complex projective variety $Y$ of dimension $n$ with the following property.
For any integer $2\leq i\leq n-1$ and any subgroup $M\subset \CH^i(Y)$ with finitely generated cokernel, we have $M\otimes \Q_\ell/\Z_\ell \neq 0$  for any prime $\ell$.
\end{corollary}
\begin{proof}
    This follows directly from Theorem \ref{thm:Chow-tensor-Ql/Zl} and the projective bundle formula, applied to $Y=X\times \CP^{n-3}$.
\end{proof}

\section{Divisibility of the torsion subgroup in Milnor K-theory} \label{sec:merkurjev-conjecture}

The following result was conjectured by Merkurjev in \cite{merkurjev}; we deduce it from Voevodsky's proof of the Bloch--Kato conjecture.  
We use the result below to give another proof of Theorem \ref{thm:unramified-Milnor-divisible}. 

\begin{theorem} \label{thm:merkurjev-conjecture}
Let $\ell$ be a prime and let $F$ be a field of characteristic different from  $\ell$.
Assume further that $F$ contains all $\ell$-power roots of unity. 
Then the $\ell$-primary torsion subgroup $K^M_n(F)[\ell^\infty]\subset K^M_n(F)$ is $\ell$-divisible.
\end{theorem}

We will need the following elementary lemma.

\begin{lemma} \label{lem:A[l^infty]}
  Let  $A$ be an abelian group and let $\ell$ be a prime number.
Suppose that for any integer $m\geq 1$ the map
$A/\ell \to A/\ell^{m+1}$ induced by multiplication by $\ell^m$ is injective.
Then the $\ell$-primary torsion subgroup $A[\ell^\infty]$ of $A$ is a divisible group.
\end{lemma}

\begin{proof}
 Let   $a\in A[\ell^\infty]$  with $\ell^m\cdot a=0 \in A$ for some $m\geq 1$. 
The image of the class of $a$ in $A/\ell$
 under the map
$$ A/\ell \to  A/\ell^{m+1}$$ given by multiplication by $\ell^{m}$ is zero.
By the injectivity assumption we conclude that $a=\ell \cdot b \in A$, for some $b$, and
$b$ satisfies $\ell^{m+1}b=0$. 
Thus any element of  $A[\ell^\infty]$ is divisible by $\ell$ in $A[\ell^\infty]$.
It follows that $A[\ell^\infty]$ is a divisible group.
\end{proof}

\begin{proof}[Proof of Theorem \ref{thm:merkurjev-conjecture}]  
Let $F$ be a field that contains all $\ell$-power roots of unity.
We aim to prove that $K^M_n(F)[\ell^\infty]$ is $\ell$-divisible.
By Lemma \ref{lem:A[l^infty]}, it suffices to prove that the multiplication by $\ell^r$ map $K^M_n(F)/\ell \to K^M_n(F)/\ell^{r+1}$ is injective.
By the Bloch--Kato conjecture, proven in \cite{Voevodsky}, the latter identifies to the map 
\begin{align} \label{eq:torsion-Milnor-1}
H^{n}(F, \mu_{\ell}^{\otimes n})  \longrightarrow  H^{n}(F, \mu_{\ell^{r+1}}^{\otimes n}) 
\end{align}
induced by the inclusion $\mu_\ell^{\otimes n} \hookrightarrow \mu_{\ell^{r+1}}^{\otimes n}$.
 
Since $F$ contains all roots of unity, we may choose a compatible system $\zeta_{r}$ of primitive
$\ell^r$-th roots of unity.
Such a choice allows us to identify the map in \eqref{eq:torsion-Milnor-1} to the map 
\begin{align*} % \label{eq:torsion-Milnor-2}
H^{n}(F, \mu_{\ell}^{\otimes n-1})  \longrightarrow  H^{n}(F, \mu_{\ell^{r+1}}^{\otimes n-1}) 
\end{align*}
induced by $\mu_\ell^{\otimes n-1} \hookrightarrow \mu_{\ell^{r+1}}^{\otimes n-1}$.
The injectivity of that map is by the long exact Bockstein sequence equivalent to the surjectivity of the reduction mod $\ell^r$ map
$$
H^{n-1}(F, \mu_{\ell^{r+1}}^{\otimes n-1}) \longrightarrow H^{n-1}(F, \mu_{\ell^{r}}^{\otimes n-1}) .
$$
The latter is in turn a direct consequence of the Bloch--Kato conjecture, proven by Voevodsky \cite{Voevodsky}, which yields canonical isomorphisms 
$H^i(F,\mu_{\ell^r}^{\otimes i})\cong K^M_i(F)/\ell^r$ for all $i$ and $r$.
This concludes the proof of the theorem.
\end{proof}

\subsection{An alternative proof of Theorem \ref{thm:unramified-Milnor-divisible}}
  
 \begin{proof}[Proof of Theorem \ref{thm:unramified-Milnor-divisible}]
 By the Chinese remainder theorem, it suffices to prove that the natural map
 $$
 H^0(X,\mathcal K^M_n)\longrightarrow  H^0(X,\mathcal K^M_n/\ell^r)
 $$
 is zero for all primes $\ell$  invertible in $k$.
 Since $k$ is separably closed, it is infinite and so $\mathcal K_n^M \stackrel{\sim}\to \mathcal H^n_M(\Z(n))$ by \cite{kerz}.
 Similarly, the Bloch--Kato conjecture proven by Rost and Voevodsky \cite{Voevodsky} together with the respective Gersten conjectures proven in \cite{BO,bloch-motivic} shows $\mathcal K_n^M/\ell^r\stackrel{\sim}\to \mathcal H^n_M(\Z/\ell^r(n))$. 
 Hence, by Theorem \ref{thm:Hi-curlyHj}, we know that for any $\alpha\in  H^0(X,\mathcal K^M_n)$, there is some positive integer $N$ such that $N\alpha$ lies in the kernel of the map in question.
 Since any natural number coprime to $\ell$ is invertible in $\Z/\ell^r$, we can without loss of generality assume that $N=\ell^s$ for some non-negative  integer $s$.
 By the Gersten conjecture for Milnor K-theory \cite{kerz} and its mod $\ell^r$-reduction \cite{BO,Voevodsky}, we have 
a commutative diagram 
$$
\xymatrix{
 H^0(X_\zar,\mathcal H_M^n(\Z(n))) \ar[r] \ar@{^{(}->}[d] &  H^0(X_\zar,\mathcal H^n_M(\Z/\ell^r(n))) \ar@{^{(}->}[d] \\
  K^M_n(k(X)) \ar[r] & K^M_n(k(X))/\ell^r 
}
$$  
where the lower horizontal map is the reduction modulo $\ell^r$ map.
Using the vertical inclusions, we can regard $\alpha$ as an (unramified) element in Milnor K-theory $K^M_n(k(X))$ and we know that for all $r$ there is some class $\beta_r\in K^M_n(k(X))$ with $\ell^s \alpha=\ell^r\beta_r$.
We then get that $\alpha-\ell^{r-s}\beta_r\in K^M_n(k(X))$ is $\ell^s$-torsion for all $r\geq s$.
Since the torsion subgroup of $K^M_n(k(X))$ is $\ell$-divisible (see Theorem \ref{thm:merkurjev-conjecture}),  we can write $\alpha-\ell^{r-s}\beta_r=\ell^{r-s}\gamma_r$ for some $\gamma_r\in K^M_n(k(X))$ and so $\alpha$ maps to zero in $K^M_n(k(X))/\ell^{r-s} $ and hence, in view of the above commutative diagram (for $r-s$ in place of $r$) to zero in $H^0(X_\zar,\mathcal H^n_M(\Z/\ell^{r-s}(n) )) $.
This holds for all $r\geq s$ and so the statement in the theorem follows.
 \end{proof}

\section{Divisibility phenomena of Lichtenbaum motivic cohomology} \label{sec:appendix-B}

Let $X$ be a smooth variety over a field $k$. 
For an abelian group $A$, we consider the motivic complex $A(j)_{\et}\in D(X_\et)$ in the \'etale site of $X$ (cf.\ Section \ref{subsec:motivic-cohomology}) and define the Lichtenbaum or {\it \'etale motivic cohomology groups} by
   $$
   H^{i}_{L}(X,A(j)): = H^{i}(X_{\et},A(j)_{\et}) .
   $$
The  abelian group structure of $H^{i}_{L}(X,\Z(j))$ for smooth projective varieties  $X$ over a separably closed field  
was discussed in papers by  Rosenschon--Srinivas \cite[Proposition 3.1]{RS-JIMJ},  Geisser \cite[Theorem 1.1]{geisser}, and Kahn \cite{kahn-divisibility}.
Kahn's paper discusses results for open smooth varieties over a separably closed field  \cite[Theorem 1.3]{kahn-divisibility}.
Geisser's paper  discusses smooth projective varieties over  separably closed fields, finite fields,  local fields and also arithmetic schemes. Over a field of characteristic $p>0$, he also considers the $p$-primary torsion of the \'etale motivic cohomology groups.

The purpose of this appendix, which does not claim originality, is as follows.
For smooth projective varieties over a separably closed field and over a finite field, we describe how the work of Geisser--Levine, Suslin, Rost, and Voevodsky, together with weight arguments \`a la Deligne lead to a precise computation of the torsion structure of  the groups $H^{i}_{L}(X, \Z(j))$ for most pairs $(i,j)$, and for $H^{i}_{M}(X, \Z(j))$ for most pairs $(i,j)$  in the  range $i\leq j+1$.
For finite fields we use the work of Kerz--Saito \cite{kerz-saito} to obtain similar results for $j\geq \dim X$ and $i\neq 2j$.
We use this to give a positive answer to Question \ref{question:intro} for some values of $(i,j)$, see Propositions \ref{prop:appendix:separable-closed} and \ref{prop:appendix-finite-field} below.

\subsection{Preliminaries} \label{sec:appendix-prelim}
We will use the following well-known facts.  
Let $X$ be a smooth variety over a field $k$.
There is a natural map $H^i_{M}(X, \Z(j)) \to H^{i}_{L}(X, \Z(j))$ which is an isomorphism after tensoring with $\Q$, see \cite[Theorem 14.24]{MVW}.
In particular the kernel and the cokernel of this map are torsion groups.
This implies that for all primes $\ell$ the map
$H^i_{M}(X, \Z(j)) \otimes \Q_{\ell}/\Z_{\ell } \to H^{i}_{L}(X, \Z(j)) \otimes \Q_{\ell}/\Z_{\ell}$
is surjective.

Let $j \geq 0$ and $i\geq 0$.  
For $X$ smooth over a field, the  map $H^i_{M}(X, \Z(j)) \to H^{i}_{L}(X,\Z( j))$ is an isomorphism  if $i\leq j+1$,  and it is injective if $i=j+2$, see e.g.\ \cite[Proof of Corollary 1.4]{AS24}.
This is referred to as the integral Beilinson--Lichtenbaum conjecture, closely related to the statement $H^{n+1}_{L}(F,\Z(n))=0$ for $F$ an arbitrary field (higher Hilbert's theorem 90); see \cite[Theorem 6.6]{Voe-milnor}, \cite[Theorem 6.18]{Voevodsky}, \cite[Conjecture 1.22]{Riou}.

By the work of Geisser--Levine \cite[Theorem 1.5]{geisser-levine} (see \eqref{eq:geisser-levine}), the Bockstein sequence for \'etale motivic cohomology yields an exact sequence
$$ 0 \longrightarrow   H^{i}_{L}(X, \Z(j))/ \ell^r \longrightarrow  H^{i}_{\et}(X, \mu_{\ell^r}^{\otimes j} ) \longrightarrow  H^{i+1}_{L}(X, \Z(j))[\ell^r]  \longrightarrow  0$$
and then, after taking direct limits,
\begin{align} \label{eq:prelim-appendix-ses}
0 \longrightarrow   H^{i}_{L}(X, \Z(j))\otimes \Q_{\ell}/\Z_{\ell} \longrightarrow  H^{i}_{\et}(X, \Q_{\ell}/\Z_{\ell}(j))
\longrightarrow  H^{i+1}_{L}(X, \Z(j))[\ell^{\infty}]  \longrightarrow  0.
\end{align}

\subsection{A  lemma on abelian groups}

\begin{lemma}\label{formel}
 Let $A$ be an abelian group. Assume $A \otimes \Q_{\ell}/\Z_{\ell}=0$
  and assume that the $\ell$-primary torsion group $A[\ell^{\infty}]$ is an extension of
  a group $F$ of finite exponent by a divisible group. Then there is a natural 
  surjective map $A \to F$, compatible with the given map $A[\ell^{\infty}] \to F$ and whose kernel is the maximal $\ell$-divisible subgroup of $A$.
\end{lemma}
  \begin{proof} The exact sequence
 $$ 0 \longrightarrow   \Z \longrightarrow    \Z[1/\ell] \longrightarrow  \Q_{\ell}/\Z_{\ell} \longrightarrow  0$$
induces the following exact sequence
 $$ 0 \longrightarrow  A[\ell^{\infty}] \longrightarrow  A \longrightarrow  A \otimes \Z[1/\ell] \longrightarrow  A\otimes \Q_{\ell}/\Z_{\ell} \longrightarrow  0.$$
 Under our hypothesis, this gives the exact sequence
  $$ 0 \longrightarrow  A[\ell^{\infty}] \longrightarrow  A \longrightarrow  A \otimes \Z[1/\ell] \longrightarrow    0.$$
By assumption, there is an exact sequence
 $$ 0 \longrightarrow  B \longrightarrow  A[\ell^{\infty}] \longrightarrow  F \longrightarrow  0,$$
 with $B$ $\ell$-divisible and $F$ a group of exponent a finite power of $\ell$.
 
The arrow $A[\ell^{\infty}] \to F$ gives rise to the
following commutative diagram of exact sequences
  $$\xymatrix{
  &  0\ar[d]   &  0\ar[d]   \\
 &  B  \ar[r]\ar[d] & B  \ar[d]   &     &\\
 0 \ar[r] & A[\ell^{\infty}] \ar[r] \ar[d] & A \ar[r]\ar[d]& A \otimes \Z[1/\ell] \ar[r]\ar[d] &   0 \\
 0 \ar[r] & F \ar[r] \ar[d] & A' \ar[r]  \ar[d] & A \otimes \Z[1/\ell] \ar[r] &   0 \\
 &  0    &  0    \\
 }$$
 Since $F$ is an $\ell$-primary group of finite exponent, the
 lower 
 sequence is split in a unique way. This produces a surjective map
 $A' \to F$ whose kernel is the   $\ell$-divisible group
 $A \otimes \Z[1/\ell]$. The kernel $A''$ of the composite, surjective map $A \to A' \to F$
is an extension of $A\otimes \Z[1/\ell]$ by $B$. These two groups are
 $\ell$-divisible, hence so is $A''$. 
  \end{proof}

\subsection{Motivic cohomology of smooth projective varieties over a separably closed field}

\begin{proposition}\label{weightL}
Let $k$ be a separably closed field. Let $X$ be a 
smooth, projective, geometrically integral variety over $k$.
Let  $\ell$ be a prime invertible in  $k$ and let $i,j \geq 0$ be non-negative integers.
\begin{enumerate}[(a)]
    \item If $i \neq 2j$, then $H^{i}_{L}(X,\Z(j))  \otimes \Q_{\ell}/\Z_{\ell}=0$.
    \item For  $i \neq 2j+1$ and $i\geq 1$, we have
$$H^{i}_{L}(X,\Z(j))[\ell^{\infty}] \simeq H^{i-1}_{\et}(X, \Q_{\ell}/\Z_{\ell} (j)).$$
The group $H^{i}_{L}(X,\Z(j))[\ell^{\infty}]$ is an extension of the finite group
$H^{i}_{\et}(X,\Z_{\ell}(j))_{tors}$ by the divisible group
$H^{i-1}_{\et}(X, \Q_{\ell}(j))/H^{i-1}_{\et}(X, \Z_{\ell}(j))$.
\end{enumerate} 
\end{proposition}

  \begin{proof}
Let $k_{0}$ be a field of finite type, $k/k_{0}$ a separable closure and $G=\Gal(k/k_{0})$.
Let  $X_{0}/k_{0}$ be smooth, projective, geometrically integral.
Let  $\ell$ be a prime invertible in  $k$ and let $i,j \geq 0$ be integers.

Let  $\alpha\in H^{i}_{M}(X,\Z(j)) \otimes \Q_{\ell}/\Z_{\ell} $.
Up to replacing $k_{0}$ by a finite extension and $X_{0}$ by the corresponding base change, we may assume that   $\alpha$ is in the image of a class in $H^{i}_{M}(X_{0},\Z(j)) \otimes \Q_{\ell}/\Z_{\ell}$.
(This uses \eqref{eq:motivic=higher-Chow}.) 
The image of the composite map
$$H^{i}_{M}(X_{0},\Z(j))  \otimes \Q_{\ell}/\Z_{\ell} \longrightarrow  H^{i}_{M}(X,\Z(j))  \otimes \Q_{\ell}/\Z_{\ell}
\longrightarrow  H^{i}_{L}(X,\Z(j))  \otimes \Q_{\ell}/\Z_{\ell} $$$$
 \longrightarrow   H^{i}_{L}(X, \Q_{\ell}/\Z_{\ell}(j)) =   H^{i}_{\et}(X, \Q_{\ell}/\Z_{\ell}(j))$$
 is invariant under $G$.
 From Deligne's theory of weights \cite{deligne}, we know (see \cite[Theorem 1.5]{CTR})  that the group of $G$-invariants of $H^{i}_{\et}(X, \Q_{\ell}/\Z_{\ell}(j))$ is finite if $i\neq 2j$.
As the group $H^{i}_{M}(X_{0},\Z(j))  \otimes \Q_{\ell}/\Z_{\ell}$ is divisible, the composite map vanishes.
Thus the composite map
$$H^{i}_{M}(X,\Z(j))    \otimes \Q_{\ell}/\Z_{\ell} \longrightarrow  H^{i}_{L}(X,\Z(j))  \otimes \Q_{\ell}/\Z_{\ell} \longrightarrow  
H^{i}_{\et}(X, \Q_{\ell}/\Z_{\ell}(j) )$$
vanishes if $i\neq 2j$. 
The LHS map is onto and the RHS map is injective, see Section \ref{sec:appendix-prelim} above.
We conclude that the group
$H^{i}_{L}(X,\Z(j))  \otimes \Q_{\ell}/\Z_{\ell}$
vanishes if  $i\neq 2j$.   
 
From this and the exact sequence \eqref{eq:prelim-appendix-ses}, we deduce that for  $i\neq 2j$ there is a natural isomorphism 
$$
H^{i+1}_{L}(X,\Z(j))[\ell^{\infty}] \simeq H^{i}_{\et}(X, \Q_{\ell}/\Z_{\ell} (j)).
$$
By the Bockstein sequence in \'etale cohomology, this group is an extension of $H^{i+1}_{\et}(X,\Z_{\ell}(j))_{tors}$ by the divisible group
$H^{i}_{\et}(X, \Q_{\ell}(j))/H^{i}_{\et}(X, \Z_{\ell}(j))$, as we want.
\end{proof}

 \begin{remark} {\rm
 The idea to use coincidence of motivic cohomology and \'etale motivic cohomology 
with rational coefficients, so that one may then use  the representation of classes
over some small subfield, is  in \cite[Proof of Thm. 1.1]{geisser} and \cite[Prop. 3.1]{RS-JIMJ}.
Geisser uses inverse limits. Here we used direct limits.}
 \end{remark}

Upon use of Lemma \ref{formel}, we deduce:

\begin{proposition}\label{HiL} Let $X$ be a smooth projective variety over a separably closed field $k$.
Let $\ell$ be a prime invertible in  $k$.  Let $j \geq 0$. Assume $i\geq 1$ and  $i\neq 2j, 2j+1$.  
 \begin{enumerate}[(a)]
     \item There is a natural surjective map $H^{i}_{L}(X,\Z(j)) \to  H^{i}_{\et}(X,\Z_{\ell}(j))_{tors}$ whose kernel  $B_{\ell}$ is the maximal $\ell$-divisible subgroup of  $H^{i}_{L}(X,\Z(j))$.
     \item The  $\ell$-primary torsion subgroup  of $B_{\ell}$ is the group
$H^{i-1}_{\et}(X, \Q_{\ell}(j))/ H^{i-1}_{\et}(X, \Z_{\ell}(j))$,
which is the maximal divisible subgroup of $H^{i-1}_{\et}(X, \Q_{\ell}/\Z_{\ell}(j))$.
\item If ${\rm char}(k)=0$,  the group $H^{i}_{L}(X,\Z(j)) $
is an extension of the finite group $\oplus_{\ell}  H^{i}_{\et}(X,\Z_{\ell}(j))_{tors}$
by a divisible subgroup whose $\ell$-primary torsion is  $H^{i-1}_{\et}(X, \Q_{\ell}(j))/ H^{i-1}_{\et}(X, \Z_{\ell}(j))$.
 \end{enumerate} 
\end{proposition}

\begin{remark} {\rm
 If   $i\leq j+1$,  then $H^{i}_{M}(X,\Z(j)) \simeq H^{i}_{L}(X,\Z(j))$ and so the structure results of Proposition \ref{HiL} hold for $H^{i}_{M}(X,\Z(j)) $ in place of $H^{i}_{L}(X,\Z(j))$.} 
 \end{remark}

\begin{remark} {\rm
If   $i=j+2$,  the map  $H^{i}_{M}(X,\Z(j)) \to H^{i}_{L}(X,\Z(j))$ is injective.
For $i\neq 2j, 2j+1$, 
 the group $H^{j+2}_{M}(X,\Z(j))    \{  \ell^{\infty}  \}$
 injects into  $H^{j+1}_{\et}(X, \Q_{\ell}/\Z_{\ell}(j))$.
For $i=4$, $j=2$,  this gives an injective  map 
 $\CH^2(X) [\ell^{\infty}] \to H^3_{\et}(X,  \Q_{\ell}/\Z_{\ell}(2))$ first considered by Bloch.}
\end{remark}

\begin{proposition} \label{prop:appendix:separable-closed}
 Let $X/k$ be smooth, projective, connected of dimension $d$ over a separably closed field $k$.
Let $\ell$ be a prime invertible in  $k$.  Let $i\geq 0$ and $j\geq 2$.
If one of the following  hypotheses holds:
  \begin{enumerate}[(a)]
\item  $i>2j$,
\item  $i \leq j+1$, or
\item  $d\leq j$ and $i\neq 2j$,
\end{enumerate}
then $H^{i}_{M}(X,\Z(j)) \otimes \Q_{\ell}/\Z_{\ell} =0$.
\end{proposition}
\begin{proof}
Statement (a) is clear because $H^{i}_{M}(X,\Z(j))=0$ for $i>2j$.
Under the hypothesis $i\leq j+1$, the map
$H^{i}_{M}(X,\Z(j)) \to H^{i}_{L}(X,\Z(j))$ is an isomorphism.
Statement (b) then comes from  Proposition \ref{weightL}. 

For any smooth, connected, quasi-projective  $X$ over $k$, for $d\leq j$ and $i \leq 2j$, Suslin  \cite[Corollary 3, p.~254]{suslin}
proved that the maps
$$
H^{i}_{M}(X, \Z/\ell^r(j)) 
\longrightarrow  H^{i}_{\et}(X,\mu_{\ell^r}^{\otimes j})
$$
are isomorphisms. 
For $d\leq j$, and any $i\geq 0$, one then gets injections
$$ 
H^{i}_{M}(X,\Z(j) ) \otimes  \Q_{\ell}/\Z_{\ell} \hookrightarrow   H^{i}_{\et}(X,\Q_{\ell}/\Z_{\ell}(j)).
$$
For $X$ smooth, connected, and projective,
and  $i\neq 2j$, a weight argument as before gives that this map is zero, which proves (c). 
\end{proof}

\subsection{Motivic cohomology of smooth projective varieties over  a finite field}

\begin{lemma}
Let $X$ be a smooth, projective, geometrically integral variety over
a finite field $\F$.  
Let $\ell$ be a prime invertible in  $\F$ and let $i \geq 0$.
If  $i \neq 2j+1, 2j+2$, then $H^{i}_{L}(X, \Z(j))[\ell^{\infty}] \simeq H^{i}_{\et}(X, \Z_{\ell}(j))_{tors}$.
\end{lemma}
\begin{proof} 
Assume  $i\neq 2j, 2j+1$.  By Deligne's results on the Weil conjectures,
 the group $H^{i}_{\et}(X, \Q_{\ell}/\Z_{\ell}(j))$
is then finite  \cite[Theorem 2, p.~780]{CTSS}.
We conclude from \eqref{eq:prelim-appendix-ses} that  $$H^{i}_{L}(X, \Z(j))\otimes \Q_{\ell}/\Z_{\ell} =0$$ 
and there is an isomorphism of finite groups
$H^{i+1}_{L}(X, \Z(j))[\ell^{\infty}] \simeq H^{i+1}_{\et}(X, \Z_{\ell}(j))_{tors}$, as we want. 
\end{proof}

 Upon use of  Lemma \ref{formel}  we deduce:

\begin{proposition}\label{HiLfini}
Let $X$ be a smooth, projective, geometrically integral variety over a finite field $\F$. 
Let $\ell$ be a prime invertible in $\F$.  
If  $i \neq  2j, 2j+1, 2j+2$, the group $H^{i}_{L}(X,\Z(j)) $ is an extension of
the finite group  $H^{i}_{\et}(X,\Z_{\ell}(j))_{tors}$ by the maximal $\ell$-divisible subgroup of $H^{i}_{L}(X,\Z(j)) $, and that subgroup is uniquely $\ell$-divisible.
\end{proposition}

\begin{remark}{\rm
If $i\leq j+1$, the isomorphism $H^{i}_{M}(X,\Z(j)) \simeq H^{i}_{L}(X,\Z(j))$ implies that we may for $i\leq j+1$ replace $H^{i}_{L}(X,\Z(j))$ by $H^{i}_{M}(X,\Z(j))$ in Proposition \ref{HiLfini}.}
 \end{remark}

\begin{remark} {\rm
If   $i=j+2$,  the map  $H^{i}_{M}(X,\Z(j)) \to H^{i}_{L}(X,\Z(j))$ is injective.
Thus for $i\neq 2j, 2j+1$, the group $H^{j+2}_{M}(X,\Z(j)) [\ell^{\infty}]$ injects into  $H^{j+1}_{\et}(X, \Q_{\ell}/\Z_{\ell}(j))$.
For $i=4$, $j=2$,  this gives an injective  map 
 $\CH^2(X) [\ell^{\infty}] \to H^3_{\et}(X,  \Q_{\ell}/\Z_{\ell}(2))$ 
 into a finite group, which was  used in \cite{CTSS}.}
\end{remark}

 We also have:
 \begin{proposition} \label{prop:appendix-finite-field}
 Let $X$ be smooth, projective, connected of dimension $d$ over a finite field $\F$.
Let $\ell$ be a prime invertible in  $\F $. Let $i\geq 0$ and $j\geq 2$.
If one of the following  hypotheses holds:
\begin{enumerate}[(1)]
    \item $i>2j$,
\item $i \leq j+1$,
\item $d\leq j$ and $i\neq 2j$,
\end{enumerate} 
then $H^{i}_{M}(X,\Z(j)) \otimes \Q_{\ell}/\Z_{\ell} =0.$
\end{proposition}
\begin{proof}
As before (1) is clear because $H^{i}_{M}(X,\Z(j))=0$ if $i>2j$.
Statement (2) is a consequence  of Proposition \ref{HiLfini}.

For any smooth, connected, projective  $X$ over $\F$, for $d\leq j$ and $0 \leq i \leq 2j$, Kerz and Saito \cite[p.\ 254, Theorem 9.3]{kerz-saito} proved that the maps
$H^{i}_{M}(X, \Z/\ell^r(j)) \to H^{i}_{\et}(X,\mu_{\ell^r}^{\otimes j})$
are isomorphisms. 
For $d\leq j$, and any $i\geq 0$, one then gets injections
$ H^{i}_{M}(X,\Z(j) ) \otimes \Q_{\ell}/\Z_{\ell} \hookrightarrow   H^{i}_{\et}(X,\Q_{\ell}/\Z_{\ell}(j)).$
For $X$ smooth, connected, and projective, and  $i <2j$, a weight argument gives that $H^{i}_{\et}(X,\Q_{\ell}/\Z_{\ell}(j))$ is a finite group, hence the map is zero.
This gives  statement (3) and hence concludes the proof.
 \end{proof}

\section*{Acknowledgements}  
Thanks to Theodosis Alexandrou, Ofer Gabber and Tam\'as Szamuely for references, to Olivier Wittenberg for a useful question, and to Federico Scavia for comments. 
Lin Zhou pointed out that item \eqref{item:thm:Hi-curlyHj:1} in Theorem \ref{thm:Hi-curlyHj}, which we had originally only proven modulo torsion, follows from item \eqref{item:thm:strengthen-thm:H_M-unramified-divisible:1} in Theorem \ref{thm:strengthen-thm:H_M-unramified-divisible}.  
This project was initiated during the first-named author’s visit to Leibniz University Hannover in the autumn of 2024, supported by his Humboldt Research Award. 
The support of the Alexander von Humboldt Foundation is gratefully acknowledged.  
During the last stages of the writing, the first-named author
enjoyed the hospitality of Lodha Mathematical Sciences Institute (Mumbai).
This project has received funding from the European Research Council (ERC) under the European Union's Horizon 2020 research and innovation programme under grant agreement No 948066 (ERC-StG RationAlgic).


\begin{thebibliography}{HKLR} 


%\bibitem[Stacks24]{stacks-project}
%The Stacks project authors, {\em  The Stacks project},
% \url{https://stacks.math.columbia.edu}, 2024.


\bibitem[Ale23]{alexandrou} Th.\ Alexandrou, {\em Torsion in Griffiths groups}, arXiv:2303.04083, to appear in Algebraic Geometry.
%
%\bibitem[Ale24]{Ale-zero-cycles}
%Th.\ Alexandrou, {\em Two Cycle Class Maps on Torsion Cycles}, International Mathematics Research Notices (2024),  rnae138. 

%\bibitem[AS23]{alexandrou-schreieder} Th.\ Alexandrou and S.\ Schreieder, {\em On Bloch's map for torsion cycles over non-closed fields}, Forum of Mathematics, Sigma (2023), Vol. 11:e53 1--21.

\bibitem[AS24]{AS24} Th.\ Alexandrou and S.\ Schreieder, {\em Truncated pushforwards and refined unramified cohomology}, Advances in Mathematics \textbf{458} (2024) 109979,
\url{https://doi.org/10.1016/j.aim.2024.109979}.


\bibitem[AZ25]{alexandrou-zhou} Th.\ Alexandrou and L.\ Zhou, {\em Torsion higher Chow cycles modulo l}, Preprint 2025, arXiv:2503.20004.

\bibitem[BaS26]{balkan-Sch}
S.\ Balkan and S.\ Schreieder, {\em Cycle conjectures and birational invariants over finite fields},
 Sel.~Math.~New Ser.~\textbf{32}, 37 (2026). \url{https://doi.org/10.1007/s00029-026-01142-0}
%Preprint 2024, arXiv:2406.14438.

%\bibitem[Be82]{Be} A.\ A.\ Beilinson, {\em Letter to C. Soul\'e}, November 1, 1982.

\bibitem[BhS15]{BS} B.\ Bhatt and P.\ Scholze, {\em The pro-\'etale topology of schemes},  Ast\'erisque \textbf{369} (2015),  99--201.

\bibitem[BO74]{BO}
S.\ Bloch and A.\ Ogus, {\em Gersten's conjecture and the homology of schemes}, Ann.\ Sci.\ \'Ec.\ Norm.\ Sup\'er., \textbf{7} (1974), 181--201.

%\bibitem[Blo79]{bloch-compositio}
%S.\ Bloch, {\em Torsion algebraic cycles and a theorem of Roitman}, Compositio Mathematica \textbf{39} (1979), 107--127.

\bibitem[Blo85]{bloch-coniveau}
S.\  Bloch, {\em Algebraic cycles and values of L-functions II}, Duke Math.\ J.\ \textbf{52} (1985), 379--397.
%
\bibitem[Blo86]{bloch-motivic}
S.\ Bloch, {\em Algebraic cycles and higher K-theory},
Adv. in Math.,\ \textbf{61} (1986), 267--304. 

\bibitem[Blo94]{bloch-JAG}
S.\ Bloch, {\em The moving lemma for higher Chow groups}, J.\ Algebraic Geom.\ \textbf{3} (1994), 537--568.

\bibitem[BE96]{BE}
S.\ Bloch and H.\ Esnault, {\em The coniveau filtration and non-divisibility for algebraic cycles}, Math.\ Ann.\ \textbf{304} (1996), 303--314.
 
%\bibitem[CD16]{cd}
%D.-C.\ Cisinski and F.\ D\'eglise, {\em \'Etale motives}, Compositio Mathematica \textbf{152} (2016), 556--666. 
 
\bibitem[CT95]{CT}
J.-L.\ Colliot-Th\'el\`ene, {\em Birational invariants, purity and the Gersten conjecture}, K-theory and algebraic geometry: connections with quadratic forms and division algebras (Santa Barbara, CA, 1992), 1--64, Proc.\ Sympos.\ Pure Math.\, \textbf{58}, AMS, Providence, RI, 1995.
 
\bibitem[CTSS83]{CTSS}
J.-L.\ Colliot-Th\'el\`ene, J.-J.\ Sansuc, and C. Soul\'e, {\em Torsion dans le groupe de Chow de codimension deux}, Duke
Math.\ J.\ \textbf{50} (1983), 763--801.

\bibitem[CTR85]{CTR}
J.-L.\ Colliot-Th\'el\`ene and W.\ Raskind, {\em $\mathcal K_2$-Cohomology and the Second Chow Group}, Math.\ Ann.\ \textbf{270} (1985), 165--199. 

%\bibitem[CT93]{CTTrento} J.-L.\ Colliot-Th\'el\`ene, {\em Cycles de torsion et $K$-th\'eorie alg\'ebrique},
% in Arithmetic Algebraic Geometry, Trento, 1991, ed. E. Ballico,
%LNM 1553, Springer-Verlag, 1993.
    
%
%\bibitem[CTO89]{CTO}
%J.-L.\ Colliot-Th\'el\`ene and M.\ Ojanguren, {\em Vari\'et\'es unirationnelles non rationnelles : au-del\`a de l'exemple d'Artin et Mumford}, Invent.\ Math.\ \textbf{97} (1989), 141--158.
%
%\bibitem[CTHK97]{CTHK}
%J.-L.\ Colliot-Th\'el\`ene, R.\ Hoobler, and B.\ Kahn, {\em  The Bloch-Ogus-Gabber theorem}, Algebraic K-theory (Toronto, ON, 1996), 31--94,  Fields Inst.\  Commun., \textbf{16}, Amer. \ Math.\  Soc., Providence, RI, 1997.
% 
%  
%\bibitem[CTV12]{CTV}
%J.-L.\ Colliot-Th\'el\`ene and C.\ Voisin, {\em Cohomologie non ramifi\'ee et conjecture de Hodge enti\`ere}, Duke Math.\ J.\ \textbf{161} (2012), 735--801.

%\bibitem[Ga94]{gabber}
%O.\ Gabber, {\em Gersten's conjecture for some complexes of vanishing cycles}, Manuscripta Math.\ \textbf{85} (1994), 323--344.

\bibitem[Del71]{deligne-HodgeII}
P.\ Deligne, {\em Th\'eorie de Hodge II}, Publ.\ Math.\ I.H.\'E.S.,  \textbf{40} (1971), 5--58.
 
\bibitem[Del74]{deligne}
P.\ Deligne, {\em La conjecture de Weil, I}, Publ.\ Math.\ I.H.\'E.S., \textbf{43} (1974) 273--307.

\bibitem[Del80]{deligneII}
P.\ Deligne, {\em La conjecture de Weil, II}, Publ.\ Math.\ I.H.\'E.S., \textbf{52} (1980) 137--252. 

\bibitem[deJ96]{deJong}
J.\ de Jong, {\em Smoothness, semi-stability and alterations}, Publ.\ Math.\ I.H.\'E.S., \textbf{83} (1996) 51--93.

\bibitem[Dia21]{diaz}
H.\ Diaz, {\em Nondivisible cycles on products of very general Abelian varieties}, J.\ Algebraic Geom.\ \textbf{30} (2021), 407--432.
 
\bibitem[Eke90]{Eke}
T.\ Ekedahl, {\em On the adic formalism}, Grothendieck Festschrift, Vol.\ II, Progr.\ Math.\ \textbf{87}, Birkh\"auser, 1990, 197--218.

%\bibitem[FS02]{F-S}
%E.\ M.\ Friedlander and A.\ Suslin, {\em The spectral sequence relating algebraic K-theory to motivic cohomology}, Ann.\ Scient.\ Éc.\ Norm.\ Sup.\ \textbf{35} (2002), 773--875.

%\bibitem[Gab83]{gabber}
%O.\ Gabber, {\em Sur la torsion dans la cohomologie 1-adique d'une variété}, C.R.\ Acad.\ Sc.\ Paris \textbf{297} (1983), 179--182.


%
%\bibitem[GL00]{geisser-levine-inventiones}
%T.\  Geisser and  M.\ Levine,  {\em The K-theory of fields in characteristic p}, Invent.\  Math.\ \textbf{139} (2000), 459--493.
%
\bibitem[GL01]{geisser-levine}
T.\ Geisser and M.\ Levine, {\em The Bloch--Kato conjecture and a theorem of Suslin–Voevodsky}, J.\ reine angew.\ Math.\ \textbf{530} (2001), 55--103. 

\bibitem[Gei10]{geisser-duality}
T.~Geisser, {\em Duality via cycle complexes},
Annals of Mathematics, \textbf{172} (2010), 1095--1126.

\bibitem[Gei17]{geisser}
T.\ Geisser, {\em On the structure of  \'etale motivic cohomology}, J.\ Pure Appl.\ Algebra \textbf{221} (2017), 1614--1628.


% \bibitem[G57]{grothendieck}
% A.\ Grothendieck, {\em Sur quelques points d'alg\'ebre homologique}, Tohoku Math.\ J.\ \textbf{9} (1957), 119--221.
% 
% \bibitem[G85]{gros}
%M.\ Gros, {\em Classes de Chern et de cycles en cohomologie de Hodge-Witt logarithmique}, M\'emoires de la SMF, $2^e$ s\'erie, tome \textbf{21} (1985).
%
%\bibitem[GS88]{gros-suwa}
%M.\ Gros and S.\ Suwa, {\em La conjecture de Gersten pour les faisceaux de Hodge-Witt logarithmiques}, Duke Math.\ J.\ \textbf{57} (1988), 615--628
%

\bibitem[Hub97]{Huber}
A.\ Huber, {\em Mixed perverse sheaves for schemes over number fields}, Compositio Math.~\textbf{108} (1997), 107--121.

%
%\bibitem[Il79]{illusie}
%L.\ Illusie, {\em Complexe de de Rham-Witt et cohomologie cristalline}, Ann.\ Sci.\ l'\'ENS \textbf{12} (1979), 501--661.

\bibitem[Jan88]{jannsen}
U.\ Jannsen, {\em Continuous \'etale cohomology}, Math.\ Ann.\ \textbf{280} (1988), 207--245.

\bibitem[Jan90]{jannsen-book}
U.\ Jannsen, {\em  Mixed Motives and Algebraic K-Theory}, Lecture Notes in Mathematics 1400, Springer, 1990.

\bibitem[Jan10]{jannsen-weights}
U.\ Jannsen, {\em Weights in Arithmetic Geometry}, Japanese Journal of Mathematics \textbf{5} (2010) 73--102.
%arXiv:1003.0927.

%\bibitem[JS09]{jannsen-saito}
%U.\ Jannsen and S.\ Saito, {\em Kato homology and motivic cohomology over finite fields},  arXiv:0910.2815.
 
\bibitem[Kah09]{kahn-divisibility}
B.\ Kahn, {\em Divisibility properties of motivic cohomology}, Preprint 2009, arXiv:1801.06010.



\bibitem[Kah12]{kahn}
B.\ Kahn, {\em Classes de cycles motiviques \'etales}, Algebra \& Number Theory \textbf{6} (2012), 1369--1407.
%



%\bibitem[Kah24]{kahn24}
%B.\ Kahn, {\em An $l$-adic norm residue epimorphism theorem}, Preprint 2024, arXiv:2409.10248.

\bibitem[Ker09]{kerz}
M.\ Kerz, {\em The Gersten conjecture for Milnor K-theory},  Invent.\  math.\  \textbf{175} (2009), 1--33.

\bibitem[KS12]{kerz-saito}
M.\ Kerz and S.\ Saito, {\em  Cohomological Hasse principle and motivic cohomology for arithmetic schemes}, Publ.\ Math.\ I.H.\'E.S.\ \textbf{115} (2012),  123--183. 


%\bibitem[Kok23]{kok} K.\ Kok, {\em On the failure of the integral Hodge/Tate conjecture for products with projective hypersurfaces}, Preprint 2023, arXiv:2305.08961.

\bibitem[KZ23]{kok-zhou}
K.\ Kok and L.\ Zhou, {\em Higher Chow groups with finite coefficients and refined unramified cohomology}, Advances in Mathematics \textbf{458} (2024) 109972, \url{https://doi.org/10.1016/j.aim.2024.109972}.
%to appear in Advances in Mathematics.

%\bibitem[KZ24]{kok-zhou2}
%K.\ Kok and L.\ Zhou, {\em On the functoriuality of refined unramified cohomology}, Preprint 2024. 


\bibitem[Lev04]{levine}
M.\ Levine, {\em K-theory and motivic cohomology of schemes, I}, Preprint (2004), \url{https://www.esaga.uni-due.de/f/marc.levine/publ/KthyMotI12.01.pdf}

%\bibitem[Ma17]{Ma}
%S.\ Ma, {\em Torsion 1-cycles and the coniveau spectral sequence}, Documenta Math.\ \textbf{22} (2017), 1501--1517.

%\bibitem[Ma22]{Ma2}
%S.\ Ma, {\em Unramified cohomology, integral coniveau filtration and Griffiths group}, Ann.\ K-Theory \textbf{7} (2022), 223--236.

\bibitem[MVW06]{MVW}
C.\ Mazza, V.\ Voevodsky and C.\ Weibel, {\em Lecture Notes on Motivic Cohomology}, American Mathematical Society (AMS); Cambridge, MA: Clay Mathematics Institute, 2006.

\bibitem[Mer88]{merkurjev}
A.~S. Merkurjev, {\em Torsion in the Milnor K-groups of fields}, (Russian) Math. USSR-Sb. {\bf 59} (1988), no.~1, 95--112; translated from Mat. Sb. (N.S.) {\bf 131(173)} (1986), no.~1, 94--112, 127. %; MR0868603


\bibitem[MS82]{MS}
A.~S.\ Merkurjev and A.~A.\ Suslin, {\em $K$-cohomology of Severi--Brauer varieties  and norm residue homomorphism},
Izv.\ Akad.\ Nauk SSSR \textbf{46} (1982), 1011--1146.

 
%\bibitem[Mil80]{milne}
%J.S.\ Milne, {\em \'Etale cohomology}, Princeton University Press, Princeton, NJ, 1980.

\bibitem[Mor25]{morel}
S.\ Morel, {\em Mixed $\ell$-adic complexes for schemes over number fields}, Doc.\ Math.\ \textbf{30} (2025), 105--181.
 
%\bibitem[Qui73]{quillen}
%D.\ Quillen,  {\em Higher algebraic K-theory, I},  Lecture Notes in Mathematics \textbf{341}, Springer, Berlin, 1973, 77--139.

\bibitem[Par96]{paranjape}
K.~H.~Paranjape, {\em Some Spectral Sequences for Filtered Complexes and Applications}, J.~Algebra \textbf{186} (1996), 793--806. 
 
\bibitem[Riou14]{Riou} J. Riou, 
La conjecture de Bloch--Kato (d'apr\`es M. Rost et V. Voevodsky),  
S\'eminaire Bourbaki, Volume 2012/2013,  SMF,
Ast\'erisque 361, 421--463, Exp. No. 1073 (2014). 

\bibitem[RS16]{RS-JIMJ}
A.\ Rosenschon and V.\ Srinivas, {\em \'Etale motivic cohomology and algebraic cycles}, J.\ Inst.\ Math.\ Jussieu\ \textbf{15} (2016), 511--537.

%\bibitem[RS18]{RS}
%A.\ Rosenschon and A.\ Sawant, {\em Rost nilpotence and \'etale motivic cohomology}, Adv.\ Math.\ \textbf{330} (2018), 420--432.

\bibitem[RS10]{RS}
A.\ Rosenschon and V.\ Srinivas, {\em The Griffiths group of the generic abelian 3-fold}, Cycles, motives and Shimura varieties, 449--467. Tata Inst.\ Fund.\ Res., Mumbai, 2010.

\bibitem[Sca24]{scavia}
F.\ Scavia, {\em Varieties over $\overline{\Q}$ with infinite Chow groups modulo almost all primes}, J.\ London Math.\ Soc.\ \textbf{110} (2024), no. 4, Paper No. e12994, 20 pp.

\bibitem[Schoe02]{schoen-modn}
C.\ Schoen, {\em Complex varieties for which the Chow group mod n is not finite}, J.\ Alg.\ Geom.\ \textbf{11} (2002), 41--100.


\bibitem[Sch21]{Sch-survey}
S.\ Schreieder, Unramified cohomology, algebraic cycles and rationality,
in: G. Farkas et al. (eds), Rationality of Varieties, Progress in Mathematics 342, Birkhäuser (2021), 345--388.

\bibitem[Sch23]{Sch-refined}
S.\ Schreieder, {\em Refined unramified cohomology of schemes}, Compositio Mathematica, \textbf{159} (2023), 1466--1530.
%
\bibitem[Sch25]{Sch-griffiths}
S.\ Schreieder, {\em Infinite torsion in Griffiths groups},  
J.\ Eur.\ Math.\ Soc.\ \textbf{27} (2025), 2571--2601.
DOI 10.4171/JEMS/1419.
% 
%
\bibitem[Sch24]{Sch-moving}
S.\ Schreieder, {\em A moving lemma for cohomology with support}, Special volume in honour of C. Voisin, Article No. 20 (2024), 50 pages.

\bibitem[Su99]{suslin} A. Suslin, {\em Higher Chow groups and \'etale cohomology}, in: Cycles, Transfer, and Motivic Homology Theories, Annals of Math.\ Studies, Princeton University Press, Princeton, 1999.

%\bibitem[SV99]{suslin-voevodsky}
\bibitem[SV00]{suslin-voevodsky}
A.\ Suslin and V.\ Voevodsky, {\em Bloch--Kato conjecture and motivic cohomology with finite coefficient}s,  
In The arithmetic and geometry of algebraic cycles (Banff,
AB, 1998), pp. 117--189, NATO Sci. Ser. C Math. Phys. Sci. 548, Kluwer Acad.
Publ., Dordrecht, 2000.
%%
%in Cycles, Transfer, and Motivic Homology Theories, Annals of Math.\ Studies, Princeton University Press, Princeton, 1999.


\bibitem[Tot16]{totaro-chow}
B.\ Totaro, {\em Complex varieties with infinite Chow groups modulo 2}, Ann.\ of Math.\ \textbf{183} (2016), 363--375.

%\bibitem[VSF00]{v-s-f}
%V.\ Voevodsky, A.\ Suslin and E.\ M.\ Friedlander {\em Cycles, transfers and motivic homology theories}, Ann.\ of Math.\ Studies, vol \textbf{143}, Princeton Univ. Press, 2000.




\bibitem[Voe02]{voevodsky-imrn}
V.\ Voevodsky, {\em Motivic cohomology groups are isomorphic to higher Chow groups in any characteristic}, Int.\ Math.\ Res.\ Not.\ \textbf{7} (2002), 351--355.

\bibitem[Voe03]{Voe-milnor}
V.\ Voevodsky, {\em Motivic cohomology with $\Z/2$-coefficients}, Publ.\ Math.\ IHES \textbf{98} (2003), 59--104.

\bibitem[Voe11]{Voevodsky}
V.\ Voevodsky, {\em On motivic cohomology with $\Z/l$-coefficients}, Ann.\ of Math.\ \textbf{174} (2011), 401--438.


%\bibitem[Voi12]{Voi-unramified}
%C.\ Voisin, {\em Degree $4$  unramified cohomology with finite coefficients and torsion codimension  $3$ cycles}, in Geometry and Arithmetic, (C.\ Faber, G.\ Farkas, R.\ de Jong Eds), Series of Congress Reports, EMS 2012, 347--368.

\end{thebibliography}
\end{document}